\documentclass{article}
\usepackage{amsmath,  amsfonts, amsthm, latexsym, amssymb}
\usepackage{a4wide} 
\usepackage{tikz}\usetikzlibrary{patterns}

 \numberwithin{equation}{section}
\newtheorem{theorem}{Theorem}[section]
\newtheorem{definition}[theorem]{Definition}
\newtheorem{proposition}[theorem]{Proposition}
\newtheorem{remark}[theorem]{Remark}
\newtheorem{lemma}[theorem]{Lemma}
 
\newtheorem{corollary}[theorem]{Corollary}
\usepackage[sc]{mathpazo}
\usepackage{color}

\def\eps{\varepsilon}
\newcommand{\cK}{{\cal K}}
\newcommand{\cP}{{\cal P}}
\newcommand{\cPr}{{\cP_{2,\rho}^r(\R^d)}}
\newcommand{\cW}{{\cal W}}
\newcommand{\di}{\operatorname{\text{div}}}

\newcommand{\R}{{\mathbb R}}

\titlepage
\title{Brake orbits and heteroclinic connections\\ for  first order Mean Field Games}
\author{ Annalisa Cesaroni and Marco Cirant}
\date{ }

\begin{document}

\maketitle
\begin{abstract}  We  consider  first order variational MFG in the whole space, with   aggregative interactions  and density constraints,   having stationary equilibria consisting of two disjoint compact sets of distributions with finite quadratic moments. Under general assumptions on the  interaction potential, we provide  a  method for the construction of periodic in time solutions to the MFG, which oscillate between the two sets of static equilibria, for arbitrarily large periods. Moreover, as the period increases to infinity, we show that these periodic solutions converge, in a suitable sense, to heteroclinic connections. As a model example, we consider a MFG system where the interactions are of (aggregative) Riesz-type. 

\medskip

\noindent
{\footnotesize \textbf{AMS-Subject Classification}}. {\footnotesize  
91A13 % Games with infinitely many players 
37J50 %  Action-minimizing orbits and measures 
37J45 % Periodic, homoclinic and heteroclinic orbits; variational methods, degree-theoretic methods 
49Q20 %Variational problems in a geometric measure-theoretic setting 
35Q91 %  PDEs in connection with game theory, economics, social and behavioral sciences 
}
\end{abstract} 
 
\tableofcontents

\section{Introduction} 
Mean field games  (MFG) theory describes interactions among  a large numbers of indistinguishable rational individuals, in which a generic agent  optimizes  some functional depending both on its dynamical state and on the average collective behavior, represented by the density of the overall population. In an equilibrium regime, the optimal dynamics   of the average agent is consistent with the collective evolution. Such  equilibria  can be described  by a system of coupled PDEs, a backward Hamilton-Jacobi equation characterizing 
the value function of the average agent, and a forward continuity equation modelling  the evolution of the  population density, that is (in the model case of first order MFG with quadratic Hamiltonian)
\begin{equation}\label{mfg} \begin{cases} 
-\partial_t u + \frac{|\nabla u|^2}{2} = f(x,m), & \text{in $(0,T) \times \R^d$} \\ 
\partial_t m-\di(m \nabla u)=0 & \text{in $(0,T) \times \R^d$} \\ 
m\geq 0, \ \int_{\R^d} m(t,x) dx=1.
\end{cases} \end{equation} Usually the system is coupled with initial/final time conditions. 
This theory has been introduced in the mathematical community by Lasry and Lions in \cite{LL061, LL062} and since then, there has been a large development of the subject in the literature. 

Here, we will focus on the widely studied class of potential (or variational \cite{var}) MFG:  these are MFG systems that can be derived as optimality conditions of suitable optimal control problems on the continuity equation, with quadratic Lagrangian and running cost $f$.  Precisely, we assume that  $f(x,m)$ is the derivative of a potential $\cW$ defined on the space of Borel probability measures $\cP(\R^d)$, that is $f(x,m)= \frac{\delta}{\delta m} \cW (m) \in C(\R^d \times \cP(\R^d))$, or equivalently
\[\lim_{h\to 0^+} \frac{\cW(m+h(m'-m))-\cW(m)}{h}= \int_{\R^d} f(x, m)d(m'-m)(x) \] for all $m,m' \in \cP(\R^d)$.  In this case, 
the  PDE system  \eqref{mfg}  formally appears as the first order condition for critical points of the following energy functional:
\begin{equation}\label{energyintro}  J_T(m,v):= \int_0^T\int_{\R^d} \frac{1}{2}|v(t,x)|^2m(t, dx) dt +\int_0^T \cW(m(t))dt, \end{equation} 
to be computed among all possible evolutions of the mass distributions, that is among all couples  $(m,v)$ such that  $m_t-\di(mv)=0$ in the distributional sense, where $m(t) \in \cP(\R^d)$ for all $t$, and the velocity field $v\in L^2(dt\otimes m(t, dx))$.  
It is well-known that when $\cW\equiv 0$, and $m(0), m(T)$ are given, this is  the so-called fluid mechanics formulation of the Monge-Kantorovich mass transfer problem introduced by Benamou and Brenier \cite{bb}, which leads  to the dynamic characterization of the $L^2$-Kantorovich-Rubinstein-Wasserstein distance $d_2$ between measures in $\cP_2(\R^d)$ (those with finite quadratic moments in $\R^d$, see Definition \ref{wdef} and \cite{ags, Sbook} for a general discussion).  
The similarities between the Benamou-Brenier formulation of optimal transport and MFG have been already explored in the study of first order MFG systems, and we refer to \cite{gs18, gmst, ls18,  ops}.  

We will construct in this work (constrained) critical points of $J_T$, rather than produce solutions to PDE systems like \eqref{mfg}. We then show that these critical points $(\bar m, \bar v)$ give rise to mean field Nash equilibria, in the following sense: for any admissible competitor $(m, v)$,
\begin{equation}\label{linez}
\int_0^T\int_{\R^d} \frac{1}{2}|v|^2m +\int_0^T f(\bar m) m \ge \int_0^T\int_{\R^d} \frac{1}{2}|\bar v|^2\bar m +\int_0^T f(\bar m) \bar m,
\end{equation}
that is, $(\bar m, \bar v)$ is a minimizer of a {\it linearized} functional. In other words, $(\bar m, \bar v)$ is such that $(\bar m, \bar v)$ itself is the optimum in a (infinite dimensional / McKean-Vlasov) control problem involving a quadratic Lagrangian and  running cost $f(\bar m)$; such a fixed point property is at the core of Nash equilibria in MFG. In addition, following  \cite{cas16}, these equilibria yield  solutions to PDE systems of the form \eqref{mfg}. This can be shown by exploiting an optimization problem in duality with \eqref{linez}. The derivation and the analysis of first order optimality conditions in PDE form will not be carried out here; besides, our constructions of critical points of $J_T$ does not require $\cW$ to have a derivative with respect to $m$, though this assumption is crucial if one wants to make a connection with MFG systems.
% (see Remark \ref{remsimm} for further considerations). 

Another observation which is crucial for this work is that variational MFG systems of the form \eqref{mfg} can be interpreted as Hamiltonian systems on the infinite dimensional metric space $\cP_2(\R^d)$, endowed with the distance $d_2$. In addition, the energy $J_T(m,v)$ defined  in \eqref{energyintro} can be rewritten via the Benamou-Brenier formula \cite{bb} as an energy on the space of trajectories $C([0, T],\cP_2(\R^d))$,  as follows:
 \begin{equation}\label{en2} J_T(m)= \int_0^T  \frac{1}{2} |m'|^2(t)  + \mathcal{W}(m(t)) \, dt,\end{equation} 
 where $|m'|(t)$ is the  metric derivative of the curve  with respect to the Wasserstein distance $d_2$, see \cite{ags}. In such a form, $J_T$ is reminiscent of standard action functionals appearing in Hamiltonian mechanics. Let us make a short remark on the sign in front of the ``potential'' $\cW$, to avoid possible confusion. In the form \eqref{en2}, $J_T$ formally corresponds to an action functional of a mechanical system of (infinitely many) particles subject to the acceleration $\nabla_x  \delta_m \cW$. In the classical mechanics terminology, the ``potential energy'' would then correspond to $-\cW$. We have adopted here the plus sign in line with the MFG viewpoint, where $\cW$ is the ``potential'' (anti-derivative) of the running cost $f$.

We will also make use, as in the work by Benamou and Brenier,  of the standard change of variables which replaces velocity by momentum, i.e. $(m,w)= (m, vm)$. The energy  \eqref{energyintro} then becomes, in a generalized sense,
\begin{equation}\label{enintro2} J_T(m,w) =\int_0^T\int_{\R^d} \frac12\left|\frac{dw}{dt\otimes m(t,dx)}\right|^2 m(t,dx)dt +\int_0^T \cW(m)dt, \end{equation} to be computed on the set   \begin{eqnarray}\nonumber  \cK &:=&\left\{(m,w)\ | \ m\in  C(\R, \cP_2(\R^d)),\right. 
\\&&  w \text{ is a  Borel $d$-vector measure  on $\R\times \R^d$, absolutely continuous w.r.t.  
 $dt \otimes m(t,dx)$, }\nonumber \\
&&    -\partial_t m + {\rm div}(w) = 0 \ \text{in the sense of distributions,   }  \nonumber \\
& &\left.  \int_{t_1}^{t_2} \int_{\R^d}\frac12 \left|\frac{dw}{dt\otimes m(t,dx)}\right|^2 m(t,dx)dt < \infty \ \text{for all $-\infty < t_1 < t_2 < \infty$}\right\}.\label{kappanuovo}\end{eqnarray} 
The two   energies \eqref{energyintro} and \eqref{enintro2} are equivalent, see \cite{var}. Note that under these new variables   the differential constraints become linear, that is $m_t-\di w=0$, and  moreover the function  $(m,w)\mapsto \frac{|w|^2}{2m}$ (extended to $0$ where $m=0$) is a convex function. 
In the following, we are going to consider constrained minimizers of \eqref{enintro2}, i.e. minimizers in some suitable subset of $\cK$. 
\smallskip
  
An interesting issue in MFG is the description of the long time behavior of equilibria, that is: given some information of the system at initial and final time, say at $t = 0$ and $t = T$, such as the population distribution $m$ and/or the final cost $u$, is it possible to describe $m$ (and $u$) at intermediate times? A natural goal would be to characterize attractors that are approached by $m$ as $T \to \infty$.
A large part of the literature in this direction is devoted to congestion type games, that are games in which players prefer sparsely populated areas of the state space.  This  is typically translated into the assumption that $\cW$ is convex, or equivalently that the interaction cost   $f(x,m)$ is  monotone increasing    with respect to the mass distribution:
\begin{equation}\label{aggrw}
\int_{\R^d}( f(x,m)- f(x,m'))d(m-m')(x)\geq 0\qquad \forall m,m' \in \cP_2(\R^d).
\end{equation}
We point out that this condition does not imply that the  functional $\cW$ is geodesically convex in $\cP_2$ (see \cite{ags, Sbook}): geodesic convexity of $\cW$ and monotonicity of $f$ are actually unrelated conditions.  We refer to \cite{alpaca} for a recent work on potential MFG with  geodesically convex  Lagrangians. 
Under the monotonicity assumption \eqref{aggrw}, one expects in general uniqueness of the equilibria, and some further regularity properties. For second order problems, the long time behavior of the PDE system is quite well understood (at least when the state space is the flat torus): in a long time horizon, solutions approach the (unique) stationary equilibrium, which is then attractive for the evolutive system. %Moreover, the stationary equilibrium is provided by the unique minimizer of $\cW$.
We refer to the recent papers \cite{cp19, cp20} and references therein for more details. 

On the other hand, without the monotonicity assumption, the long time behavior is much less understood and very few is known about long time patterns. The second author obtained recently some results  for viscous (second order) MFG  in the flat torus with anti-monotone interactions, that is assuming that $-f(x,m)$ is monotone increasing. 
In particular in \cite{c19} (see also \cite{cn18}),  it is provided the construction, using bifurcation arguments, of an infinite number of branches of non-trivial solutions which exhibit an oscillatory (in time) behavior, and emanating from a trivial stationary solution  (also for the case of two populations of players, which is non-variational in general). Finally, in \cite{masoero19}, by using weak KAM methods in an infinite dimensional setting, it is provided an example of a second order MFG with non monotone interaction cost, settled in the periodic torus, for which solutions in the long time horizon do not  converge to the stationary state (see also \cite{carda19} for further results). Long time pattern formation has also been explored in MFG models arising in socioeconomics \cite{hongler2018, ullmo2017, YMMS}. For first order problems, the long time behavior is even less understood (we are only aware of few results in \cite{CGra}, regarding the monotone case).

\smallskip
 
In this paper, we analyze long time patterns arising in some first order (potential) MFG. Differently with respect to previous works, our setting is first order (no viscosity), on the whole space and without periodicity conditions.
Moreover, we consider a non-monotone case, namely $\cW$ will be a sort of ``double-well'' potential.
We have in mind models where players aim at aggregating, that is, they attracted towards crowded areas, see in particular Section \ref{subsi} below. %: \begin{equation*}  \int_{\R^d}( f(x,m)- f(x,m'))d(m-m')(x)  \leq 0 \qquad \forall m, m'.
%\end{equation*}
%Anti monotonicity  ensures some compactness to the problem, and so we expect existence of  solutions in the whole space $\R^d$, without periodicity conditions.
In contrast with the aggregating forces, we impose density constraints to the population density, that are particularly meaningful when one describes crowd motions for example. We indeed impose the distribution of players $m(t)$ to have a density which does not exceed some given value $\rho$, that is, for all $t$,
\[
m(t) \in \cPr := \{m \in \cP_2(\R^d) \ : \ \exists \ 0 \le \tilde m \le \rho \ \text{ a.e. on $\R^d$ s.t. $m = \tilde m dx$}\}
\] 
(with a slight abuse of notation, we will often identify $m$ with its density $\tilde m$) and so we restrict the set $\cK$ defined in \eqref{kappanuovo} to 
\begin{equation}\label{kappas}\cK^\rho := \{(m,w)\in \cK,  \ m(t) \in \cPr, \ \forall t\in \R\}.
\end{equation} This constraint models an environment with finite capacity. Alternatively, it could be regarded as an infinite cost paid by players that try to cluster over saturated regions (hard congestion). We mention that first order MFG with density constraints have been studied, in the monotone case, in \cite{cas16}, where connections with variational models for the incompressible Euler's equations \`a la Brenier are also discussed (see also \cite{ls18}). Another effect against concentration could be dissipation, that may appear as a viscosity term in the continuity equation for $m$. This setting has been considered recently in \cite{cc}, where stationary solutions to  second order aggregating MFG are constructed; concentration phenomena and selection type results when the dissipation term vanishes are also shown. 

Throughout the paper, we assume the following general conditions on the interaction potential $\cW : \cPr \to [0, \infty)$. Note that ``aggregation'' is not encoded explicitly in our setting. Nevertheless, if $\cW$ has to be double-well shaped in the sense described below, then $\cW$ has to break the convexity assumption (typically related to ``competition''), and therefore it has to enforce ``aggregation'' to some extent (see the model in Section \ref{subsi}). First of all we assume that $\min_{\cPr} \cW$ exists, and without loss of generality that  $\min_{\cPr} \cW = 0$. We suppose in addition that minima of $\cW$ consists of {\it two disjoint compact subsets} of $\mathcal{P}_2(\R^d)$, that is 
\begin{equation}\label{z} \tag{{\bf Z}} \exists \mathcal{M}^+, \mathcal{M}^- \subset\subset \cPr \ \text{ s.t. } \ d_2(\mathcal{M}^+, \mathcal{M}^-)
=: 2q_0>0, \ \text{ and }\ \cW(m)=0 \Leftrightarrow m\in \mathcal{M}^\pm.\end{equation}
(where $\mathcal{M}^\pm = \mathcal{M}^+ \cup \mathcal{M}^-$). We assume some standard lower semi-continuity (in a topology which is slightly weaker than the one of $\cP_2$)
\begin{equation}\label{lsc} \tag{{\bf lsc}}
\text{ for any $p<2$, $\{m_n\} \subset \cPr$, \ if \ } \lim_n d_p(m_n, m)=0 \ \text{ then } \  \liminf_n\cW(m_n)\geq \cW(m),
\end{equation}  
which will be needed to construct minimizers of \eqref{enintro2}. Note that lower semi-continuity of the kinetic part term in $J_T$ is standard by convexity (see Proposition \ref{lsck}). %It is sufficient for $\cW$  to have lower semicontinuity  in  every Wasserstein space $\cP_p(\R^d)$ for  $p<2$, see Definition \ref{assw}. 
%Note that $\cP_p(\R^d)\subseteq \cP_2(\R^d)$ and $d_p(m,n)\leq d_2(m,n)$ for every $p<2$. 
Some coercivity of $\cW$ in $\cP_2(\R^d)$ will be also needed: there exists $C_\cW>0$ such that for all $m \in \cPr$
\begin{equation}\label{b} \tag{{\bf BDD}} 
-C_\cW+C_\cW^{-1} \int_{\R^d} |x|^2 m(x)dx \leq  \cW(m)\leq C_\cW\left( 1+  \int_{\R^d} |x|^2 m(x)dx \right).
\end{equation}
%where $d_2(m, \mathcal{M}^\pm)=\min (d_2(m, \mathcal{M}^+), d_2(m, \mathcal{M}^-))$. 
  Note that \eqref{b} implies that $\cW$  has compact sublevel sets in $\cP_p(\R^d)$ for every $p<2$, see Lemma \ref{equiconv}  and Remark \ref{remequiconv},  but not necessarily for $p=2$.  
  
We finally assume the following continuity property in $\cP_2(\R^d)$ close to the zero level-set: for any $\{m_n\} \subset \cPr$,
\begin{equation}\label{c}\tag{{\bf CON}}
\text{ if \quad $\lim_n\cW(m_n)=0$, \quad  then \quad $\lim_n d_2(m_n, \mathcal{M}^\pm)=0$}.
\end{equation}
Note that if $\cW$  is assumed to be lower semicontinuous and with  compact sublevel sets in $\cP_2(\R^d)$, then \eqref{c}  follows directly from \eqref{z}. 

\smallskip

It is clear that minima $\mathcal{M}^\pm$ of $\cW$ are stationary solutions/equilibria, namely minimizers of the energy $J_T$. The main goal of this work is show that the MFG problem has other equilibria that exhibit peculiar patterns. First, we construct {\it periodic in time} critical points of $J_T$, that oscillate between stationary solutions (brake orbits). Then, we construct {\it heteroclinic connections}, that are, roughly speaking, solutions to the MFG problem which are defined for all times and approach $\mathcal{M}^-$ as $t \to -\infty$ and  $\mathcal{M}^+$ at $t \to +\infty$ (see Definition \ref{hetero}). We will exploit the fact that the potential $\cW$ in the energy \eqref{enintro2} is assumed to be a double-well potential  in $\cPr$. Written in the form \eqref{en2}, the energy can be interpreted as an action functional on the space of continuous curves with values in the metric space $\cP_2(\R^d)$, and is reminiscent of classical variational problems for finite-dimensional Hamiltonian systems.

There is a huge literature (see  the survey \cite{rabi} and references therein) on the construction of periodic or heteroclinic  trajectories in Hamiltonian systems by means of variational techniques. 
%Among periodic solutions, the so-called brake orbits are widely studied; these are $T$-periodic  curves $m^T$ such that 
% \begin{equation}\label{br} m^T\left(\frac{T}4-t\right)=m^T\left(\frac{T}4+t\right), \end{equation} so a brake orbit basically travels along the same trajectory back and forth in $T/2$-time (note that the speed $m'\left(\pm T/4\right)$  vanishes).  Brake orbits are periodic critical points of the action functional \eqref{en2} (with Morse index $1$ in the context of periodic perturbations) and not global minimizers. To mode out this instability, some symmetry can be added to the system. Here, we assume  that there exists a reflection $\gamma:\R^d\to \R^d$, such that \begin{equation}\label{rif} \tag{{\bf REF}} \cW (\gamma_\# m)=\cW (m) \qquad \text{and} \qquad \mathcal{M}^+ = \gamma_\# \mathcal{M}^-,  \end{equation}
%that is, $\cW$ is symmetric and vanishing on two disjoint subsets that are symmetric with respect to each other.
Among periodic solutions, the so-called {\it brake orbits} are widely studied. These are solutions $m$ of the Hamiltonian system in  $[t_0, t_1]$ such that $|m'|(t_0)=|m'|(t_1)=0$, which are extended to  periodic curves with period $2|t_1-t_0|$ just by reflection around $t_0$ or $t_1$.   Roughly speaking, a brake orbit travels periodically between $m(t_0)$ and $m(t_1)$, and $m(t_0)$ and $m(t_1)$ are typically close to steady states.
%(note that the speed $m'\left(\pm T/4\right)$  vanishes). 
These orbits are usually found as periodic critical points of the action functional \eqref{en2} (with Morse index $1$ in the context of periodic perturbations), and not as global minimizers. To mode out their instability (and to circumvent the lack of compactness given by time translations in \eqref{en2}) some symmetry can be added to the system. Here, we assume  that there exists a reflection $\gamma:\R^d\to \R^d$, such that \begin{equation}\label{rif} \tag{{\bf REF}} \cW (\gamma_\# m)=\cW (m) \qquad \text{and} \qquad \mathcal{M}^+ = \gamma_\# \mathcal{M}^-,  \end{equation}
that is, $\cW$ is symmetric and vanishing on two disjoint subsets that are symmetric with respect to each other. To generate brake orbits, we then restrict to those trajectories that satisfy 
\begin{equation}\label{br} m(-t) = \gamma_\# m(t), \qquad m\left(\frac{T}4-t\right)=m\left(\frac{T}4+t\right) \qquad \forall t,
\end{equation}
so that $|m'|(T/4)=|m'|(-T/4)=0$. In this way, we construct orbits that follow back and forth the same trajectory connecting $m(-T/4)$ and $m(T/4)$ in time $T/2$, thus performing a whole cycle in time $T$. 

\smallskip
 
Before stating our results, we recall that other  extensions to the  infinite dimensional setting of  these kind of constructions has been considered quite recently  in the literature. 
The existence of heteroclinic connections in the general framework of metric spaces has been provided in \cite{ms}, under the assumption that the potential $\cW$ has  a finite numbers of zeros. The result is  obtained by a  different procedure, namely by re-parametrizing the action functional \eqref{en2} to a length functional in the metric space: then an heteroclinic connection is  a geodesic with respect the new length functional.  
%, to be lower semicontinuous and finally to have compact sublevel sets in the set of all supports of curves where the length functional is well defined. 
Another class of infinite dimensional problems which have been considered in the literature  is related to functionals $\cW$ defined on Hilbert spaces (such as $H^1(\Omega)$, with appropriate boundary conditions) and $\cW(u)=\|\nabla u\|_{L^2(\Omega)}^2+ \int_\Omega W(x,u)dx$, where $W(x, \cdot)$ is a double well potential. In \cite{am1} (see also references therein)   the authors prove the existence of  brake orbits and also convergence to heteroclinic connections  as the period goes to infinity  by minimizing the action functional  among curves with prescribed energy.  Analogous  results have been proved in \cite{fgn}, with a different approach: 
instead of minimizing the action functional  with fixed mechanical energy, the author  minimize it  on a set of $T$ -periodic maps with fixed $T > 0$.  
In this paper, we follow the same approach as in \cite{fgn},  and as far as we know, similar constructions for MFG problems have never been studied. 

The first main result is about construction of brake orbits, and it is proved in Section \ref{secper}.  
% First of all, due to the density constraints,  we may restrict to consider $m$ which have a density (with respect to Lebesgue measure), which we still denote, with an abuse of notation, as $m$. The same will be done with $w$. 
We introduce the sets of curves on which we minimize our functional
\begin{align} \cK_T^{\rho, S}:=  \Bigl\{ (m,w) \in \cK^\rho, \ T\text{-periodic}, 
\ m\left( \frac T4 + t\right) = m\left( \frac T4 - t\right), \, m(-t) = \gamma_\# m(t), \, \forall t \in \R \Bigr\}\label{kapparhos}.
\end{align}
%where $\cK^\rho$ is defined in \eqref{kappas}. Observe that, due to the symmetry assumption \eqref{rif}, we have that the two symmetry  conditions $m\left( T/4 + t\right) = m\left( T/4 - t\right), \, m(-t) = \gamma_\# m(t)$ appearing in \eqref{kapparhos} 
%are natural, in the sense that minimizers in $\cK_T^{\rho, S}$ are also critical points  see Remark \ref{remsimm}. 
Then, we have the following result.
\begin{theorem}\label{teo2} Assume \eqref{z}, \eqref{lsc},  \eqref{b}, \eqref{c} and \eqref{rif}. 
Let $q \in (0,  q_0)$, where $q_0$ is defined in \eqref{z}. Then  there exists $\bar T=\bar T(q) > 4$   such that, for any $T \ge \bar T$, there exists a $T$-periodic minimizer $(m^T, w^T)\in\cK_T^{\rho, S}$ of the problem $\min_{\cK_T^{\rho, S}} J_T$. Moreover, $(m^T, w^T)$ satisfies
\begin{equation}\label{vicino}
\begin{cases}
d_2(m^T(t), \mathcal{M}^+) < q & \forall t \in \left(s, \frac T2- s\right) \\
d_2(m^T(t),\mathcal{M}^-) < q & \forall t \in \left(-\frac T2 + s,  - s\right),
\end{cases}
\end{equation}
for some $0<s =s(q)$ (note that $s(q)$ does not depend on $T$). 
% Moreover $\bar T(q)\to +\infty$ as $q\to 0$. 
\end{theorem} 
As a consequence, we show in Corollary \ref{remsimm} that there exists a brake orbit for the MFG problem in the sense of Definition \ref{bo}, that is: a Mean Field Nash equilibrium, i.e.  a minimizer as in \eqref{linez}
in the larger non-symmetric set
\begin{equation}  \cK_T^{\rho }:=  \Bigl\{ (m,w) \in \cK^\rho, \ T\text{-periodic}\Bigr\}\label{kapparho}
\end{equation}
that exhibit (symmetric) oscillations between two nearly-steady states. Note that $\cK_T^{\rho, S}$ contains no information on the ``closeness'' of $m$ to $\mathcal M^\pm$. What we prove is that minimizers of $J_T$ in $\cK_T^{\rho, S}$ have a posteriori to be close to $\mathcal M^\pm$, in the sense of \eqref{vicino}. Therefore, these nearly-steady states are close to be minimizers of the potential energy $\cW$.

\smallskip
We stress that the transition time $2s$ between (neighborhoods of) the two steady states depends on $q$ only, and remains bounded as $T \to \infty$. This is a key point in obtaining the second main result, which is about the construction of heteroclinic solutions and convergence of brake orbits to heteroclinics (that will proved in Section \ref{sechet}). Heteroclinics are equilibria (minimizers as in \eqref{linez}), connecting different sets of steady states in an infinite time horizon: 
 $\lim_{t\to -\infty} d_2(m(t), \mathcal{M}^-)=\lim_{t\to +\infty} d_2(m(t), \mathcal{M}^+)=0$ (see Definition \ref{hetero} below).

To this aim, we introduce the energy on  the whole space:
\[
J(m,w)=\int_{-\infty}^{+\infty} \int_{\R^d} \frac12\left|\frac{dw}{dt\otimes m(t,dx)}\right|^2 m(t,dx)dt +\int_{-\infty}^{+\infty}\cW(m(t)) \, dt,
\] 
 and the sets of curves 
\begin{equation}\label{kappainftyintro}
\cK^{\rho, S}:= \left\{ (m,w) \in \cK^\rho : \   \ m(-t) = \gamma_\# m(t)\ \text{ for all } t,\  J(m,w)<+\infty\right\}. 
\end{equation} Note that (see Lemma \ref{lemmat}), if $(m,w)\in \cK^{\rho,S}$,  then \[\lim_{t\to\pm\infty}d_2\left( m(t), \mathcal{M}^\pm\right)=0.\] 
% Moreover, arguing as in Remark \ref{remsimm},  minimizers in $\cK^{\rho, S}$ are   also  critical points in $\cK^\rho$. 

We have the following result. 
\begin{theorem}\label{teo3}Assume \eqref{z}, \eqref{lsc},  \eqref{b}, \eqref{c} and \eqref{rif}. 
\begin{itemize}\item[a)] There exists a minimizer $(m,w)\in \cK^{\rho, S}$ of the problem $\min_{\cK^{\rho,S}} J$. 

\item[b)] For any $T > 0$, let $(m^T, w^T)\in\mathcal{K}_T^{\rho, S}$ be a minimizer of $J_T$ constructed in Theorem \ref{teo2}. 
Then
\[\lim_{T\to +\infty}d_2^2\left( m^T\left( \frac{T}{4}\right), \mathcal{M}^+\right)=0=\lim_{T\to +\infty}d_2^2\left( m^T\left( -\frac{T}{4}\right),\mathcal{M}^-\right), \] 
and up to passing  to subsequences $T_n\to +\infty$, there holds 
\[
\begin{split}
&m^{T_n}\to \hat m\qquad \text{locally uniformly in $C(\R, \cP_p(\R^d))$ for all $p<2$} \\
&w^{T_n}\to \hat w \qquad \text{ weakly in $L^2([-L,L]\times \R^d)$ for all $L>0$},
\end{split}
\] 
and
\[
\begin{split}
& \text{$(\hat m,\hat w) \in\mathcal{K}^{\rho, S} $ is a minimizer of the problem $\min_{\cK^{\rho,S}} J$},  \\
& J(\hat m,\hat w)=\frac{1}{2}\lim_{T\to +\infty} J_T(m^T, w^T).
\end{split}
\]
\end{itemize}
\end{theorem} 
As a consequence, we obtain that minimizers obtained above are heteroclinic connections for the MFG in the sense of Definition \ref{hetero}, see Corollary \ref{ultimoc}.

\smallskip

We make a few final remarks in light of the two results. As we previously mentioned, the unique minimizer of $\cW$ is an attractor of MFG equilibria under the monotonicity assumption \eqref{aggrw}. If one drops \eqref{aggrw}, the picture may then change substantially. Heteroclinics produced here connect two different minimizers of $\cW$; hence, the state of the system can be arbitrarily close to a minimum (with respect to $d_2$) of $\cW$, and converge to a different steady state as $t \to \infty$. A further study of stability of minimizers of $\cW$ can be matter of future work.

 Note again that brake orbits and heteroclinics obtained in Theorems \ref{teo2} and \ref{teo3} yield solutions to MFG systems of the form \eqref{mfg} in a suitable weak sense. The connection between the variational formulation and the PDE system for first order problems has been extensively studied in \cite{cas16}, and adaptions to our framework may require minor technical work. 
 
 We finally note that to obtain Theorems \ref{teo2} and \ref{teo3} it is not really necessary to work with absolutely continuous measures that are bounded in $L^\infty$, that is on $\cPr$. Using the same techniques in Wasserstein spaces, it would be possible to restate them in $\cP_2$ (reformulating accordingly all the assumptions). We present anyway our general results in $\cPr$ to accomodate the following ``aggregation'' model, that is a main motivation of this work. Without the density constraint, such a model would become probably more trivial, and surely more far from describing motions of crowds (because of likely formation of singularities). Similarly, it is not really necessary to set up everything on the whole euclidean space: analogous results could be proven on bounded domains or  for periodic in space probabilities. Still, we wanted to modify the setting of previous works, to show that non-trivial long time patterns may arise in a non-periodic environment (and without the presence of viscosity).

%We mention that, in addition to ``pressure'' terms appearing in the Hamilton-Jacobi equation (due to density constraints) and an ergodic constant (due to periodicity on $[0,T]$) no further multipliers related to $m(T/4 + t) = m(T/4  -t)$, $m(-t) = \gamma_\# m(t)$ appear in view of the symmetry assumption \eqref{rif}.  for additional considerations on these constraints.

%It is also possible to show that minimizers $(m,w)$ obtained in Theorems \ref{teo2} and \ref{teo3} provide solutions to MFG systems of the form \eqref{mfg} in a suitable weak sense. The connection between the variational formulation and the PDE system for first order problems has been extensively studied in \cite{cas16}, and the adaption to our framework may require minor technical work. We mention that, in addition to ``pressure'' terms appearing in the Hamilton-Jacobi equation (due to density constraints) and an ergodic constant (due to periodicity on $[0,T]$) no further multipliers related to $m(T/4 + t) = m(T/4  -t)$, $m(-t) = \gamma_\# m(t)$ appear in view of the symmetry assumption \eqref{rif}. See  Remark \ref{remsimm} for additional considerations on these constraints.
 
\subsection{A model problem}\label{subsi}  Finally, we present a model problem where our results apply. We consider a vatiational MFG where the potential term $\cW$   is given by 
 \begin{equation}
\label{enintrosta}\mathcal{W}(m)=\int_{\R^d} W(x)m(dx)-\int_{\R^d}\int_{\R^d} K(|x-y|)m(dx)m(dy).
\end{equation}  Note that in this case $f(x,m)= \delta_m \cW = W(x)-2\int_{\R^d}K(|x-y|)m(dy)$. 
The first part of  the energy is a potential energy, where  $W:\R^d\to [0, +\infty)$ is a ``double-well'' confining function, symmetric under a reflection $\gamma$, vanishing on two disjoint  balls $B_1, B_2$ (with $B_1 = \gamma(B_2)$), and  quadratically increasing at infinity: see the assumption \eqref{assw}.  The function $W(\cdot)$ models the spatial preference of the agents. 
The second part of the energy is an interaction energy, modeled through the interaction kernel $-K$.  $K$ is assumed to be positive definite, 
%so that $f(x,m)$ satisfies assumption \eqref{aggrw}.  Moreover $K$ is 
radially symmetric,  locally integrable and increasing at zero (in an appropriate sense), see 
\eqref{assK} and \eqref{pos}. In particular a model class of  such interaction kernels $K$ is given by the Riesz kernels 
\begin{equation}\label{Riesz}
K(|x-y|)=\frac{1}{|x-y|^{d-\alpha}},\text{  \quad with $\alpha\in (0,d)$. }
\end{equation}
For instance, \eqref{enintrosta} can be associated to a crowd of agents that aim at minimizing their reciprocal distances (with spatial preference $W$).  Note that the {\it aggregation} effect plays against the hard density constraint on $m$: no further aggregation is possible whenever $m = \rho$.
Energies like \eqref{enintrosta} have been recently studied extensively, as they are directly connected to a class of self-assembly/aggregation models which have received much attention, see e.g. \cite{cho} and references therein. 

It is possible to show, see Section \ref{secmodel}, that under the previous assumptions on $W, K$, $\cW$ defined in \eqref{enintrosta} satisfies \eqref{b}, \eqref{lsc}, \eqref{c}, \eqref{rif}. Regarding the general assumption \eqref{z}, we are indeed able to provide a full description of minimizers of $\cW$, and therefore our results apply. Informally speaking, the description of minimizers can be stated as follows. For precise statements, see Theorem \ref{thmminima} and Proposition \ref{classiminimi}.

\begin{theorem}\label{teointro1}  Under the assumptions that W is symmetric, vanishing on two disjoint balls and with quadratic growth, and $K$ behaves like \eqref{Riesz} (i.e. \eqref{assw}, \eqref{ref}, \eqref{assK}, \eqref{pos}), there exist minimizers of \eqref{enintrosta} in $\cPr$,
% \begin{equation}\label{kappaintro} \mathcal{K}_\rho= \{m\in \cP_2(\R^d), \ m\leq \rho\}\end{equation}  
and all the minimizers are given by characteristic functions (multiplied by $\rho$) of compact sets in $\R^d$. 

If, in addition, the flat zones of minima of $W$ are sufficiently large in terms of $\rho$ (i.e. \eqref{ass3} holds), then all the minimizers consists of two compact disjoint sets $\mathcal{M^\pm}\subset \cPr$, symmetric with respect to each other, whose elements are characteristic functions of balls (multiplied by $\rho$). 
\end{theorem} 
 
So, in the case described in Theorem \ref{teointro1},  Theorems \ref{teo2} and Theorem \ref{teo3} apply and we may construct brake orbits and heteroclinic solutions. \smallskip

Some interesting issues, in our opinion, are left open for this model problem. In particular, we know by Theorem \ref{teointro1} that stationary minimal solutions to the MFG problem are given by characteristic functions, i.e. $m=\rho\chi_E$, where $E$ is a compact set. A natural  question is whether or not minimizers (or even critical points) of the evolutive energy \eqref{enintro2}  enjoy these two features; that is, time-dependent equilibria have compact support and are evolving characteristic functions. In other words, given a periodic brake orbit $(m^T, w^T)$ as in Theorem \ref{teo2}, or a minimal heteroclinic connection $(m, w)$ as in Theorem \ref{teo3}, is it true that $m^T(t), m(t)$ are  characteristic functions of a family of evolving compact sets $E_t$ for all times ? At the moment a full answer to this question seems far to be reached.  

Another natural related problem is the version  of the game with a large (but finite) number of players:  MFG can be interpreted (and derived) indeed as limiting models for large populations of interacting agents, where any given individual is affected by the averaged state of the other individuals. In the companion work \cite{ccprep},  we consider the analogous variational problem involving the energy \eqref{energyintro} of a finite number of interacting particles, where the density constraint appears as a bound from below on the minimal distance between particles (being in turn inversely proportional to the number of particles $N$). First of all, we formalize the connection between the discrete $N$-particles problem and the continuous MFG model by proving a $\Gamma$-convergence type result, as $N\to +\infty$, of the energies, in the same spirit of  \cite{flos}. Moreover, we show that for the $N$-particle system, at least in the $1$-dimensional case, periodic minimizers are compactly supported, and particles minimize reciprocal distances. This gives a partial answer to our question (again, at least in dimension one), namely we provide the existence of limiting brake orbits for the continuous problem that are time-dependent characteristic functions. 

 \subsection*{Acknowledgements} The authors wish to thank the anonymous referees for many comments and suggestions, which were useful to improve the overall presentation of the paper. The authors are partially supported by the Fondazione CaRiPaRo Project ``Nonlinear Partial Differential Equations: Asymptotic Problems and Mean-Field Games'', and are members of GNAMPA-INdAM. The second author has been partially supported by the Programme ``FIL-Quota Incentivante'' of University of Parma and co-sponsored by Fondazione Cariparma.

 \subsection*{Notation}  We wil denote by $B(x,r)\subseteq \R^d$ the ball centered at $x$ and with radius $r$, $B_r= B(0,r)$  and $\omega_d=|B_1|$. 
For any  measurable set $E\subseteq \R^d$, we define $\chi_E$ to be the characteristic function of $E$.  

$\cP(\R^d)$, $\cP_p(\R^d)$ and $\cPr$ are (sub)sets of Borel probability measures defined below (see Section \ref{secassumption}). For any set $\mathcal{M} \subset \cP_2(\R^d)$, $d_2(\mu, \mathcal{M})=\inf_{m \in \mathcal{M} } d_2(\mu,m)$.  

%Finally, for $m\in \cP_2(\R^d)$, we say that $m\leq \rho$, if $m$ has a density, still denoted with $m$, such that $m\in L^\infty(\R^d)$ with $\|m\|_\infty\leq \rho$.  If $(m,w)\in \mathcal{K}$ and $m\leq \rho$, we also identify $w$ with its density with respect to $dt\otimes m dx$ and the kinetic part of the energy in \eqref{enintro2} reads for all $s<t$, 
%\[
%\int_{s}^{t} \int_{\R^d} \frac12\left|\frac{dw}{d\tau\otimes m(\tau,dx)}\right|^2 m(\tau,dx)d\tau=
%\int_{s}^{t}\frac12 \int_{\R^d}  \left| \frac{ w(\tau, x)}{ m(\tau,x)}\right|^2   m(\tau,x) \, dx d\tau.
%\]

\section{The Wasserstein spaces}  \label{secassumption} We introduce some notions for calculus in Wasserstein spaces that will be useful in the following. For a general reference on these results we refer to \cite{ags}, \cite{Sbook}. First, let $ \mathcal{P}(\R^d)$ be the space of Borel probability measures on $\R^d$, endowed with the topology of narrow convergence, that is:
 \begin{definition}  Let $\mu_k, \mu\in \mathcal{P}(\R^d)$.  We say that $\mu_k\to \mu$ narrowly if 
 \begin{equation}\label{narrowly}
 \lim_k \int_{\R^d}g(x)\mu_k(dx)= \int_{\R^d} g(x) \mu(dx)\qquad \forall g\in C_b(\R^d),
 \end{equation}
where $C_b(\R^d)$ is the space of continuous and bounded functions on $\R^d$.  
 \end{definition} 
 Note that this notion of convergence is equivalent to the one of convergence in the sense of distributions (see \cite[Remark 5.1.6]{ags}), where one has \eqref{narrowly} for all smooth and compactly supported test functions $g\in C^\infty_0(\R^d)$.
 
\begin{definition}\label{wdef}  
Let $p\geq 1$. The Wasserstein space of Borel probability measures  with bounded $p$-moments is defined by
\[\mathcal{P}_p(\R^d)=\left\{\mu\in \mathcal{P}(\R^d)\ \Big| \int_{\R^d} |x|^p d\mu(x)<+\infty\right\}.\] 

The Wasserstein space can be endowed with the $p$-Wasserstein distance
\begin{equation}\label{wassdis} d^p_p(\mu, \nu)= \inf\left\{\int_{\R^d} \int_{\R^d} |x-y|^p d\pi(x,y)\ |\ \pi\in \Pi(\mu, \nu)\right\}\end{equation} 
where $\Pi(\mu,\nu)$  is the set of  Borel probability measures $\pi$ on $\R^d\times \R^d$ such that $\pi(A\times\R^d)=\mu(A)$ and $\pi(\R^d\times A)=\nu(A)$ for  any Borel set $A\subseteq \R^d$.  
%When $p=1$ and $\mu, \nu$ have compact support,  $1$-Wasserstein distance has also the following dual representation
%\begin{equation} \label{1dist} d_1(\mu, \nu)=\sup\left\{\int_{\R^d} \phi(x)d(\mu-\nu)(x)\ |\ \phi:\R^d\to \R, \text{Lip}(\phi)\leq 1\right\}.
%\end{equation}
 \end{definition} 
 
 Note that $\mathcal{P}_q(\R^d)\subset \mathcal{P}_p(\R^d)$ for $p<q$, and by Jensen inequality, $d_p(\mu, \nu)\leq d_q(\mu, \nu)$ for $p<q$.  We then recall the following results about narrow convergence and convergence in Wasserstein spaces. 
\begin{lemma}\label{lemmaconv} Let $\mu_k, \mu\in \mathcal{P}(\R^d)$ such  that $\mu_k$ converges to $\mu$ narrowly. \begin{enumerate}
\item Let $g:\R^d\to [0, +\infty]$ be lower semicontinuous. Then  \[\liminf_k \int_{\R^d} g(x)d\mu_k(x)\geq \int_{\R^d} g(x) d\mu(x). \] 
\item Let $g:\R^d\to [0, +\infty)$, continuous and $\mu_k$-integrable, be such that \[\limsup_k \int_{\R^d} g(x)d\mu_k(x)\leq \int_{\R^d} g(x) d\mu(x)<\infty.\] 
Then, $g$ is uniformly integrable with respect to $\mu_k$, that is
\[\lim_{R\to +\infty} \sup_{k}\int_{\{x\ | g(x)\geq R\}}  g(x)d\mu_k(x)=0.\]\end{enumerate} 
\end{lemma} \begin{proof}We refer to \cite[Lemma  5.1.7]{ags}.
\end{proof} 
\begin{lemma}\label{equiconv}
  $\mathcal{P}_p(\R^d)$ endowed
with the p-Wasserstein distance is a separable complete  metric space.  A set $\mathcal{M}\subset \mathcal{P}_p(\R^d)$  is relatively compact if and only if it has uniformly integrable  $p$-moments, that is 
 \[\lim_{R\to +\infty} \sup_{\mu\in \mathcal{M}}\int_{\R^d\setminus B(0, R)} |x|^pd\mu(x)=0.\] 
Let  now $\mu_k,\mu\in \mathcal{P}_p(\R^d)$ for some $p\geq 1$. Then the statements below are equivalent:
\begin{enumerate}
\item $d_p(\mu_k, \mu)\to 0$
\item $\mu_k$ converges to $ \mu$ narrowly and  $\mu_k$ have uniformly integrable $p$-moments.
\end{enumerate} 

Finally, for any $\nu\in \cP_p(\R^d)$, the map $\mu\to d_p(\mu, \nu)$ is lower semicontinuous with respect to narrow convergence. 

\end{lemma}
\begin{proof} We refer to  \cite[Prop. 7.1.5]{ags}. Note that if $\mathcal{M}$ has uniformly integrable $p$-moments then it is tight, i.e. for all $\eps>0$ there exists 
 $K_\eps\subseteq \R^d$  compact for which  $\sup_{\mu\in \mathcal{M}}\int_{\R^d\setminus K_\eps}d\mu(x)\leq \eps$. 
 
 The lower semicontinuity of the Wasserstein distance is proved in \cite[Proposition 7.1.3]{ags}. 
\end{proof}
\begin{remark}\label{remequiconv} \upshape 
Note that, if for some $q>p$, 
 \[  \sup_{\mu\in \mathcal{M}}\int_{\R^d} |x|^q d\mu(x)<+\infty \] 
 then $\mathcal{M}$ has uniformly integrable $p$-moments. 
\end{remark} 

 Finally, we introduce a subspace of regular measures as follows. 
 \begin{definition}\label{defreg} We define  $ \cPr$ to be the set of measures belonging to $\cP_2(\R^d)$ and having density in $L^\infty(\R^d)$, with $L^\infty(\R^d)$-norm bounded by $\rho > 0$:
\[
  \cPr = \{m \in \cP_2(\R^d) \ : \ \exists \ 0 \le \tilde m \le \rho \ \text{ a.e. on $\R^d$ s.t. $m = \tilde m dx$}\} .
 \]
  \end{definition} 
  
Since we will work with measures with density in $L^\infty$, we recall here the notion (or characterization) of weak-* convergence in $L^\infty$: for $\mu_k, \mu \in L^\infty(\R^d)$, $\mu_k$ is said to converge to $\mu$ weak-* in $L^\infty$ if  \[\lim_k \int_{\R^d}g(x)\mu_k(dx)= \int_{\R^d} g(x) \mu(dx)\qquad \forall g\in L^1(\R^d).\]

We now make a few considerations on the kinetic part of the energy in \eqref{enintro2}, that is on the functional
\[(m, w) \mapsto\int_{t_1}^{t_2}\int_{\R^d} \frac12\left|\frac{dw}{dt\otimes m(t,dx)}\right|^2 m(t,dx)dt ,\] 
which can be defined in general for couples $(m,w)$, $m \in C(\R, \cP_{1}(\R^d))$ and $w$ a Borel $d$-vector measure  on $\R\times \R^d$, absolutely continuous w.r.t.  $dt \otimes m(t, dx)$. These properties are indeed part of the definition of admissible couples $(m,w)\in \cK$. Throughout the paper, $m(t)$ will further satisfy the $L^\infty$ constaint $m(t) \in \cPr$. We immediately note that if $(m,w)\in \cK^\rho$, then $w$ has a density which is in $L^2_{\rm loc}(\R, L^2(\R^d))$, that is
\[
\int_{t_1}^{t_2}\int_{\R^d} |w|^2dxdt=\int_{t_1}^{t_2} \int_{\R^d} \left| \frac{w(t,x)}{m(t,x)}\right|^2 m^2(t,x) \, dx dt\leq \rho\int_{t_1}^{t_2} \int_{\R^d} \left| \frac{w(t,x)}{m(t,x)}\right|^2 m(t,x) \, dx dt.
\] 
Moreover, by H\"older inequality and recalling that $m(t)\in\cP_2(\R^d)$, we have
\[
\left(\int_{t_1}^{t_2}\int_{\R^d} |x| |w(t,x)|dxdt\right)^2 \leq \left(\int_{t_1}^{t_2}\int_{\R^d} |x|^2 m(t,x)dxdt\right)\left(\int_{t_1}^{t_2}\int_{\R^d}  \left| \frac{w(t,x)}{m(t,x)}\right|^2 m(t,x) \, dx dt\right) < \infty.
\]

We now state a lower semi-continuity result (which could be stated for weaker convergence in the variables $m, w$, but it will be used below in the present form).

\begin{proposition}\label{lsck} Suppose that $m_n \to m$ in $C(\R, \cP_{p}(\R^d))$ for some $p \ge 1$, $w_n \rightharpoonup w$ (weakly) in $L^2((-\infty, \infty) \times \R^d)$, and $m_n(t), m(t)$ are absolutely continuous with respect to the Lebesgue measure for all $t, n$. Then,
\[
\int_{t_1}^{t_2} \int_{\R^d} \left|\frac{w(t,x)}{m(t,x)}\right|^2 m(t,x)dxdt \le \liminf_{n \to \infty} \int_{t_1}^{t_2} \int_{\R^d} \left|\frac{w_n(t,x)}{m_n(t,x)}\right|^2 m_n(t,x)dxdt.
\] 
\end{proposition}
 \begin{proof} See    \cite[Proposition 5.18]{Sbook}. \end{proof}
 Finally, we recall the following  uniform continuity property of elements belonging to $\mathcal{K}$, that will be useful in the sequel. 
 \begin{proposition} Let $(m,w)\in \mathcal{K}$, as defined in \eqref{kappanuovo}. Then 
\begin{equation}\label{unmezzoh}
\big(d_2(m(t), m(s))\big)^2 \le (t-s)  \int_{s}^{t} \int_{\R^d}\left|\frac{dw}{dt\otimes m(\tau,dx)}\right|^2 m(\tau,dx)d\tau, \qquad \forall t > s. % \in (t_1, t_2).
\end{equation} 
 \end{proposition} 
 \begin{proof}  This can be proved using H\"older inequality and \cite[Thm 8.3.1]{ags}. 
 \end{proof}

\section{Brake periodic solutions} \label{secper}

Let us first define brake orbits. For the sake of clarity, let us summarize first our constraint sets, from larger to smaller:
\[
\begin{split}
& \cK \, : \, \text{the set of admissible flows of mass-momentum $(m,w)$, def. in \eqref{enintro2}} \, \supset \\
& \cK^\rho \, : \, \text{flows such that $m(t)$ is bounded in $L^\infty(\R^d)$, def. in \eqref{kappanuovo}} \, \supset \\
& \cK^\rho_T \, : \, \text{$T$-periodic in time flows, def. in \eqref{kapparho}} \, \supset \\
& \cK^{\rho,S}_T \, : \, \text{$T$-periodic and symmetric flows, def. in \eqref{kapparhos}}.
\end{split}
\]
\begin{definition}\label{bo} Let $ (\bar m, \bar w) \in \mathcal{K}^\rho$. We say that $(\bar m, \bar w)$ is a {\bf brake orbit} for the MFG if $ (\bar m, \bar w) \in \cK^{\rho,S}_T$ and
\begin{multline*}
\int_0^T \int_{\R^d}  \left| \frac{\bar w(t,x)}{\bar m(t,x)}\right|^2 \bar m(t,x) + f(x,\bar m(t)) \bar m(t,x) dx dt \le \\ \int_0^T \int_{\R^d}  \left| \frac{w(t,x)}{m(t,x)}\right|^2 m(t,x) + f(x,\bar m(t)) m(t,x) dx dt \qquad \forall (m,w)\in \mathcal{K}^\rho_T.
\end{multline*}
\end{definition} 
In other words, a brake orbit is a periodic flow enjoying symmetries, which is an Nash equilibrium in a suitable Mean Field sense, i.e. a minimizer of an infinite-dimensional control problem among {\it non-symmetric} competitors.

\smallskip
We now turn to the proof of Theorem \ref{teo2}. 
Denote by $Q = (-\infty, \infty) \times \R^d$ and, for any $T > 0$, $C_T((-\infty, \infty), \cP_2(\R^d))$ be the subset of $T$-periodic elements of $C((-\infty, \infty), \cP_2(\R^d))$. 
We provide a preliminary result on the existence of constant speed geodesics in $\cK^\rho$ based on displacement convexity introduced by McCann \cite{mccann}. 
\begin{lemma}\label{connettori} Let $t_1 < t_2$, and $m_1, m_2 \in \cPr$.
% be compactly supported and satisfying $0 \le m_i(x) \le \rho$ a.e. in $\R^d$. 
 Then, there exists a couple $(m, w) \in \cK^\rho$, as defined in \eqref{kappas}, such that   
\[
m(t_1) = m_1, \qquad m(t_2) = m_2, \qquad \int_{\tau_1}^{\tau_2} \int_{\R^d} \left| \frac{w(t,x)}{m(t,x)}\right|^2 m(t,x) \, dx dt = \frac{d^2_2(m(\tau_1), m(\tau_2))}{\tau_2 - \tau_1} .
\]
for all $t_1 \le \tau_1 \le \tau_2 \le t_2$.
\end{lemma}

\begin{proof}
% Let $B \subset \R^d$ be a ball containing the supports of $m_1$ and $m_2$ and 
Let   $ \hat m$ be  the unique constant speed geodesic $ \hat m \in AC([0, 1], \cP_2(\R^d))$ (see \cite[Section 7.2]{ags}) connecting $m_1$ and $m_2$ (i.e. $\hat m(0) = m_1$, $\hat m(1) = m_2$), which satisfies for all $0 \le s_1 \le s_2 \le 1$
\[
d_2(\hat m(s_1), \hat m(s_2)) = (s_2-s_1)d_2(m_1, m_2),
\]
that is,
\[
|\hat m'|(s) = d_2(m_1, m_2) \qquad \text{for all $s \in (0,1)$.}
\]

The functional $m \mapsto \|m\|_{L^q(\R^d)}$ is {\it geodesically convex} in $\cP_2(\R^d)$ for every $q>1$ (see \cite[Proposition 7.29]{Sbook}), namely  $\hat m(s)$ is in $L^q(\R^d)$ for every  $q> 1$ and every $s$, and it satisfies
\[
\| \hat m(s)\|_{L^q(\R^d)} \le \max\{\|m_1\|_{L^q(\R^d)}, \|m_2\|_{L^q(\R^d)} \} \leq \max\{\|m_1\|_{L^\infty(\R^d)}^{1-1/q}, \|m_2\|_{L^\infty(\R^d)}^{1-1/q} \} \leq \rho^{1-1/q}.
\] So, by Chebycheff inequality, for all $c>0$ and all $q>1$ we have $|\{x\in \R^d: \ \hat m(s)>c\}|^{1/q}\leq  c^{-1}\rho^{1-1/q} $, and therefore
$\hat m(s)\in L^\infty (\R^d)$ with $\| \hat m(s)\|_{L^\infty (\R^d)} \leq \rho$.

Being $\hat m(t)$ a constant speed geodesic connecting $m_1$ and $m_2$ ,
\begin{equation}\label{dm1m2}
d^2_2(\hat m(s_1), \hat m(s_2)) = (s_2-s_1)^2 \big(d_2(m_1, m_2)\big)^2= (s_2-s_1)\int_{s_1}^{s_2} |\hat m'|^2(s) \, ds  \qquad \text{for all $0 \le s_1 \le s_2 \le 1$.}
\end{equation}
 By \cite[Thm 8.3.1]{ags} we get for a.e. $s\in(0,1)$ the existence of a vector field $\hat v(s) \in L^2(\hat m(s); \R^d)$ such that $-\partial_t \hat m + {\rm div}(\hat v\, \hat m) = 0$
is satisfied in the distributional sense, and for a.e. $s$
\[
|\hat m'|(s) = \left(\int_{\R^d} |\hat v(s, x)|^2 \hat m(s,x) dx \right)^{1/2}.
\]
Hence, substituting $|\hat m'|(s)$ into \eqref{dm1m2} and setting $\hat w = \hat v \hat m$ (on the set $\{m > 0\}$, and identically zero elsewhere), we obtain
\[
d^2_2(\hat m(s_1), \hat m(s_2)) = (s_2-s_1)\int_{s_1}^{s_2}  \int_{\R^d} \left| \frac{\hat w(s,x)}{\hat m(s,x)}\right|^2 \hat m(s,x)\, dx ds.
\]
We then have that the couple $(\hat m, \hat w)$ belongs to $\cK$. To obtain the required couple $(m, w)$ it is enough to perform a linear change of variables, i.e.
\[
m(t,x) := \hat m\left(\frac{t-t_1}{t_2-t_1}, x\right),\quad w(t,x) := \frac 1 {t_2-t_1} \hat w\left(\frac{t-t_1}{t_2-t_1}, x\right).
\]
Finally we extend $m(t,x)$ to all $t\in\R$ by setting $m(t,x)=m_1(x)$, $w(t,x)=0$  for $t<t_1$,  and $m(t,x)= m_2(x), w(t,x)=0$ for $t>t_2$.
\end{proof}

Now we need a technical lemma about positivity properties of the functional $\cW$ outside $\mathcal{M}^\pm$. 
\begin{lemma}\label{lemmaw0} For any $q>0$, we have
\begin{equation}\label{delta}
\inf\left\{ \cW(m)\ |\ m\in \cPr,  \ d_2(m, \mathcal{M}^\pm)\geq q\right\}= : \delta(q, \cW) >0.
\end{equation}  
\end{lemma} 
\begin{proof}
Assume by contradiction that  there exists $q > 0$ for which $\delta=0$. We consider $m_n\in \cPr$ such that $q \le d_2(m_n,  \mathcal{M}^\pm) $ and $ 0\leq\cW(m_n)\leq 1/n$.  By the lower bound in the assumption \eqref{b} and Lemma \ref{equiconv}, we have that the sublevel set $\cW(m)\leq 1$ is compact in $\cP_p(\R^d)$, for any $p<2$, so by \eqref{lsc} and \eqref{z},   we conclude   that \[\lim_n \cW(m_n)=0.\]  Hence, by the continuity property \eqref{c} we get that $\lim_n d_2(m_n, \mathcal{M}^\pm)=0$, which is in contradiction with the fact that $d_2(m_n,  \mathcal{M}^\pm)\geq q$. \end{proof} 

The next lemma \ref{cruciallemma} is crucial, and roughly states the following: suppose that $(m,w)\in \cK^\rho$ is a (bounded energy) competitor, and that $m(t)$ is sufficiently close to  $\mathcal{M}^+$ (resp. $\mathcal{M}^-$)  at some times $t_1,t_2$. Then, if it does not remain close to $\mathcal{M}^+$ (resp. $\mathcal{M}^-$) in the whole time interval $[t_1, t_2]$, it is possible to modify it to decrease its energy. The lemma is based on a cut argument, which has been already used in  the analysis  of  periodic orbits and heteroclinic connections for Hamiltonian systems, see e.g. \cite[Lemma 2.1]{fgn}.

\begin{lemma}\label{cruciallemma}  Let  $0<t_1 < t_2<T$.  Let $(m, w) \in \cK^\rho$   and assume that $(m,w)$ satisfies
\[
\int_{t_1}^{t_2} \int_{\R^d} \left| \frac{w(t,x)}{m(t,x)}\right|^2 m(t,x) \, dx dt  \leq C' .
\]
%Suppose also that $m(t)$ has compact support for all $t \in (t_1, t_2)$. 
Let now $q \in (0, q_0]$, where $q_0$ is as in \eqref{z}. Then, there exists $q' = q' (q, C', \cW)$ such that, if the following conditions are fulfilled for some $\bar m^+\in \mathcal{M}^+$ and $t^* \in (t_1, t_2)$ : 
\[
\begin{split}
&\bullet \, d_2(m(t_1), \bar m^+) \le q' , \quad   d_2(m(t_2), \bar m^+) \le q', \\
&\bullet \,   d_2(m(t^*), \mathcal{M}^+)> q,
\end{split}
\]
 then there exists $(\mu, v) \in \cK^\rho$ with the following properties:
\[
\begin{split}
&\bullet \, \big(\mu(t), v(t)\big) = \big(m(t), w(t)\big) \quad \text{for all $t \in \R\setminus (t_1, t_2)$},\\
&\bullet \, d_2(\mu(t), \mathcal{M}^+) < q \quad \text{for all $t \in (t_1, t_2)$},\\
&\bullet \, \int_{t_1}^{t_2} \int_{\R^d} \frac{|v(t,x)|^2}{\mu(t,x)}  \, dx dt  + \int_{t_1}^{t_2} \cW(\mu(t)) \, dt < \int_{t_1}^{t_2} \int_{\R^d} \frac{|w(t,x)|^2}{m(t,x)}  \, dx dt  + \int_{t_1}^{t_2} \cW(m(t)) \, dt.
\end{split}
\]
\end{lemma}

\begin{proof} For any $0 < q' < q/2$ set
\[
\begin{split}
& \tau_1 = \max\{ t > t_1 : d_2(m(s), \mathcal{M}^+) \le q, \quad\text{for all}\quad s \le t\}, \\
& \tau'_1 = \max\{ t < \tau_1\ : d_2(m(t), \mathcal{M}^+) \le q' \}.
% \tau_2 = \min\{ t < t__2 : d_2(m(s), m_+) \le q, \qquad\text{for}\qquad s \ge t\},
\end{split}
\]
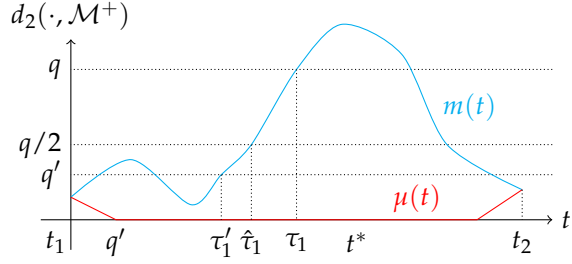
\begin{figure}\label{fig1}
\centering
 \begin{tikzpicture}[scale=2]
       \draw[->] (-.2,0) -- (3.2,0) node[right] {$t$};
      \draw[->] (0,-.2) -- (0,1.2) node[above] {$d_2(\cdot, \mathcal{M}^+)$} ;
      
       \draw[-, densely dotted] (3.2,.3) -- (0,.3) node[left] {$q'$};
       \draw[-, densely dotted] (3.2,.5) -- (0,.5) node[left] {$q/2$};
       \draw[-, densely dotted] (3.2,1) -- (0,1) node[left] {$q$};
       
       \draw [cyan] plot [smooth, tension=.5] coordinates { (0,.15) (0.4, .4) (.8,.1) (1, .3) (1.2, .5) (1.5, 1) (1.8, 1.3) (2.2, 1.1)  (2.5, .5 )(3,.2)};
        \draw [red] plot coordinates { (0,.15) (.3,0) (2.7,0)  (3,.2)};
       
       \draw[-, densely dotted] (1,.3) -- (1,0) ;
       \draw[-, densely dotted] (1.2,.5) -- (1.2,0) ;
       \draw[-, densely dotted] (1.5,1) -- (1.5,0)  ;
        \draw[-, densely dotted] (3,.2) -- (3,0)  ;
       
       \node at (1,-.15) {$\tau'_1$};
       \node at (1.2,-.15) {$\hat \tau_1$};
       \node at (1.5,-.15) {$\tau_1$};
       \node at (1.9,-.15) {$t^*$};
        \node at (.3,-.15) {$q'$};
         \node at (-.1,-.15) {$t_1$};
          \node at (3,-.15) {$t_2$};
        
         \node [red] at (2.3,.14) {$\mu(t)$};
         \node [cyan] at (2.65,.75) {$m(t)$};
         
%      \draw[scale=1.3,domain=-1.2:1.2,smooth,samples=200,variable=\x,blue] plot ({\x},{max(0, .3 - 3*\x*\x*(1-\x*\x)/(1+ \x*\x)) }) node[above] {$W(x)$};
%      \draw[-,red, very thick] (-1,1.5) -- (-.6,1.5) node[above] {$m(x)$};
%      \draw[-, red, dashed] (-.6,0) -- (-.6,1.5);
%      \draw[-, red, dashed] (-1,0) -- (-1,1.5);
%     \node at (.8,-.2) {$B(a^+, r)$};
%       \node at (-.8,-.2) {$B(a^-, r)$};
 \end{tikzpicture}
 \caption{\footnotesize The construction of the alternative competitor $(\mu, v)$. }
 \end{figure}
It holds $t_1 \leq \tau'_1 < \tau_1 < t^*<t_2$, and $q' \le d_2(m(t),  \mathcal{M}^+) \le q$ for all $t \in [\tau'_1, \tau_1]$. Note that, by \eqref{unmezzoh} and the triangle inequality,
\begin{multline*}
(C'(\tau_1 - t_1))^{1/2} \ge d_2(m(\tau_1), m(t_1)) \ge d_2(m(\tau_1), \bar m^+) - d_2(m(t_1), \bar m^+)\\ \ge d_2(m(\tau_1),  \mathcal{M}^+)-d_2(m(t_1), \bar m^+) \ge q - q'>\frac{q}{2},
\end{multline*}
hence
\[
t_2 >\tau_1> t_1 +\frac{q^2}{4C'}. 
\]

We construct $(\mu, v)$ as follows. Choose $0 < q' < \frac{1}{2}\min\left\{  \frac{q^2}{4C'},q\right\}$. By means of Lemma \ref{connettori}, there are two couples $(m_1, w_1)$ and $(m_2, w_2)$ belonging to $\cK^\rho$ which connect  $m(t_1)$ to
$\bar m^+$ at time $t_1 + q'$ and $\bar m^+$ at time $t_2 - q'$ to $m(t_2)$, respectively. Set then
\[
\mu(t) := 
\begin{cases}
m_1(t) & t \in [t_1, t_1 +  q'], \\
\bar m^+ & t \in [t_1 +  q', t_2 - q'], \\
m_2(t) & t \in [t_2 - q', t_2], \\
m(t) & \text {otherwise}
\end{cases}
\qquad
v(t) := 
\begin{cases}
w_1(t) & t \in [t_1, t_1 +  q'], \\
0 & t \in [t_1 +  q', t_2 - q'], \\
w_2(t) & t \in [t_2 - q', t_2], \\
w(t) & \text {otherwise}
\end{cases}
\]
The constraint $(\mu, v) \in \cK^\rho$ is easily verified. 
Note   that  for $t \in [t_1, t_1 +  q']$, since $\mu(t)$ is a constant speed geodesic connecting $\mu(t_1) = m(t_1)$ and $\mu(t_1 + q') = \bar m^+$,
\[
d_2(\mu(t), \mathcal{M}^+)\leq d_2(\mu(t), \bar m^+) \le d_2(\mu(t_1), \bar m^+) = d_2(m(t_1), \bar m^+)\leq  q'.
\]
The same inequalities holds on $t \in [t_2-q', t_2]$, hence \begin{equation}\label{unost} d_2(\mu(t), \mathcal{M}^+) \leq q'\leq q\qquad\text{ for all $t \in (t_1, t_2)$.}\end{equation}
Since $\mathcal{M}^+$ is compact,
\[
\int_{\R^d} |x|^2 \mu(t,x) dx = d_2(\mu(t), \delta_0) \le d_2(\mu(t), \mathcal{M}^+) + d_2(\mathcal{M}^+, \delta_0) \le \hat c + q,
\]
for some $\hat c > 0$.
 
Therefore, by Lemma \ref{connettori} and  the growth assumption on $\cW$ given by \eqref{b},  we get 
 \begin{align}\label{eqzz}
  \int_{t_1}^{t_2} & \int_{\R^d} \frac{|v|^2}{\mu}  \, dx dt  + \int_{t_1}^{t_2} \cW(\mu(t)) \, dt =\\ 
\nonumber & = \int_{t_1}^{t_1 + q'} \int_{\R^d} \frac{|v_1|^2}{m_1}  \, dx dt  + \int_{t_1}^{t_1 + q'} \cW(m_1(t)) \, dt + \int_{t_2 - q'}^{t_2} \int_{\R^d} \frac{|v_2|^2}{m_2}  \, dx dt  + \int_{t_2 - q'}^{t_2} \cW(m_2(t)) \, dt\\ 
\nonumber &=  \frac{d_2^2(m(t_1), \bar m^+)}{q'} + \int_{t_1}^{t_1 + q'} \cW(m_1(t)) \, dt +\frac{d_2^2(m(t_2), \bar m^+)}{q'}+ \int_{t_2 - q'}^{t_2} \cW(m_2(t)) \, dt  \\  
\nonumber & \qquad\qquad \leq  2q' + 2q' C_\cW (1+\hat c + q).
\end{align} 
 
We now introduce a further intermediate time $\hat \tau_1 := \max\{ t < \tau_1\ : d_2(m(t), \mathcal{M}^+) \le q/2 \}$. It holds $\tau'_1 <\hat \tau_1 < \tau_1$, and $q/2 \le d_2(m(t),\mathcal{M}^+) \le q$ for all $t \in [\hat \tau_1, \tau_1]$.

 By the triangular inequality and the compactness of $\mathcal{M}^\pm$, recalling the definition of $q_0$, we get $d_2(m(t), \mathcal{M}^-)\geq 2q_0-d_2(m(t), \mathcal{M}^+)\geq 2q_0-q\geq q$ for all  $t \in [\hat \tau_1, \tau_1]$. 
Therefore by Lemma \ref{lemmaw0}, we get that there exists $\delta=\delta (q/2, \cW)$ such that  $\cW(m(t))\geq \delta>0$ for all $t \in [\hat \tau_1, \tau_1]$. 
%By \eqref{unmezzoh} 
%\[
%C'(\tau_1 - \hat \tau_1) \ge d_2^2(m(\tau_1), m(\hat \tau_1)) .
%\]
Recall that $(m, w) \in \cK$, so \cite[Theorem 8.3.1]{ags} guarantees that $ \int_{\R^d} \frac{|w(t)|^2}{m(t)} dx \ge \big(|m'|(t)\big)^2$ for a.e. $t$. Hence, by Young's inequality and the triangle inequality,
\begin{multline*}
\int_{t_1}^{t_2} \int_{\R^d} \frac{|w|^2}{m}  \, dx dt  + \int_{t_1}^{t_2} \cW(m(t)) \, dt \ge \int_{\hat \tau_1}^{\tau_1} \int_{\R^d} \frac{|w|^2}{m}  \, dx dt  + \int_{\hat \tau_1}^{\tau_1} \cW(m(t)) dt  \\
\ge \sqrt{2} \int_{\hat \tau_1}^{\tau_1} \left(\int_{\R^d} \frac{|w|^2}{m}  \, dx \right)^{1/2} \sqrt{\cW(m(t))}\, dt \ge \sqrt{2 \delta} \int_{\hat \tau_1}^{\tau_1} |m'|(t) \, dt = \sqrt{2 \delta} \, d_2(m(\hat \tau_1), m(\tau_1))
\\ \ge  \sqrt{2 \delta} \big( d_2(m(\tau_1), \mathcal{M}^+) - d_2(m(\hat \tau_1), \mathcal{M}^+) \big) =  \frac q 2 \sqrt{2 \delta}.
\end{multline*}
Combining this inequality with $\eqref{eqzz}$ we complete the proof of the lemma, decreasing eventually $q'$ so that $2q' + 2q' C_\cW (1+\hat c + q)<  \frac q 2 \sqrt{2 \delta}$.
\end{proof} 

We are now ready to construct $T$-periodic minimizers of $J_T$. We restrict the class $\cK$ to flows of probability measures that  are $T$-periodic and enjoy additional symmetries, so we introduce the set $\cK_T^{\rho, S}$ as defined in \eqref{kapparhos}. 
We observe that the second symmetry constraint $m(-t)=\gamma m(t)$ rules out orbits which remain for all time in $\mathcal{M}^+$ or in $\mathcal{M}^-$. The first symmetry constraint $m(T/4+t)=m(T/4-t)$ is due to the fact that we are looking for brake periodic orbits, which oscillate twice in a period between $\mathcal{M}^+$ and $\mathcal{M}^-$. Note that we are using the notation $\gamma m(t) = \gamma_\# m(t)$; since $m(t)$ has a density, this means that $(\gamma m)(t,x) = m(-t,\gamma(x))$ a.e..

We provide now the proof of the first main result, that is Theorem \ref{teo2}.  

\begin{proof}[Proof of Theorem \ref{teo2}]
{\it Step 1: Energy bounds.} Choose any $m_0 \in \cPr$ with compact support such that $m_0 = \gamma m_0$. 
Observe that by \eqref{rif} we have $ \gamma \mathcal{M}^+ = \mathcal{M}^-$. Since $d_2$ is preserved by the transformation $\gamma$, we can define  \[h : = d_2(m_0, \mathcal{M}^+)=d_2(m_0,  \mathcal{M}^-).\]  Let $\bar m_+\in \mathcal{M}^+$, such that 
$d_2(m_0, \mathcal{M}^+)=d_2(m_0, \bar m_+)$. So $d_2(m_0, \mathcal{M}^-)=d_2(m_0, \gamma \bar m_+)$. By Lemma \ref{connettori}, there exists a couple $(m, w) \in \cK^\rho$  that connects $m_0$ at time  $t = 0$ to
$\bar m_+$ at time $t = 1$.   Let $T > 4$, and for $t \in [0, T/2]$,
\[
\tilde m(t) := 
\begin{cases}
m(t) & t \in [0, 1], \\
\bar m_+ & t \in [1, T/2 - 1], \\
m(T/2-t) & t \in [T/2-1, T/2]
\end{cases}
\qquad
\tilde w(t,x) := 
\begin{cases}
w(t,x) & t \in [0, 1], \\
0 & t \in [1, T/2 - 1], \\
-w(T/2-t,x) & t \in [T/2-1, T/2].
\end{cases}
\] Observe that $d_2(\tilde m(t), \mathcal{M}^+)\leq d_2(\tilde m(t), \bar m_+)\leq h$ for all $t\in [0, T/2]$. 
On the interval $[-T/2, 0]$, $(\tilde m, \tilde w)$ can be extended symmetrically: \[(\tilde m(t), \tilde w(t)):=(\gamma \tilde m(-t), -\gamma \tilde w(-t)),\]
Finally, $(\tilde m, \tilde w)$ can be extended periodically over the whole time interval, so $(\tilde m, \tilde w) \in \cK_T^{\rho, S}$. 

Moreover we compute, recalling Lemma \ref{connettori}, and the growth condition \eqref{b} on $\cW$,
\begin{multline}\label{defc}
0\leq J_T(\tilde m, \tilde w) = 4 \int_0^1 \int_{\R^d} \left| \frac{\tilde w(t,x)}{\tilde m(t,x)}\right|^2 \tilde m(t,x) \, dx dt  +4 \int_0^1 \cW(\tilde m(t)) \, dt 
\\ \le 4d^2+ 4C_\cW(1 + d^2_2(\tilde m(t), \delta_0)) \le 4d^2+ 4C_\cW\big(1 + (h + d_2(\mathcal{M}^+, \delta_0))^2\big) =: C'.
\end{multline}Note that $C' > 0$ does not depend on $T$. We may then suppose that along any minimizing sequence $(m_n,w_n)$,
\begin{equation}\label{enbound}
\int_0^T \int_{\R^d} \left| \frac{w_n(t,x)}{m_n(t,x)}\right|^2 m_n(t,x) \, dx dt \le J_T(m_n, w_n) \le C'.
\end{equation}

{\it Step 2: Minimizing sequences can be chosen to be close to  $\mathcal{M}^\pm$.} Pick any minimizing sequence $(m_n, w_n) \in \cK_T^{\rho,S}$ of $J_T$. 
 Fix now $n \in \mathbb N$. Let $ q \in (0, q_0]$, and $0 < q' < q$ be as in Lemma \ref{cruciallemma} (with $C'$ as in \eqref{enbound}). 

Note that the triangle inequality, the invariance of $d_2$ under $\gamma$, $m_n(0)=\gamma m_n(0)$, $\mathcal{M}^+ = \gamma \mathcal{M}^-$ imply  that $d_2(\mathcal{M}^+ ,  m_n(0))=d_2(\mathcal{M}^- ,  m_n(0))$ and then \[
2q' < 2q_0 = d_2(\mathcal{M}^+ , \mathcal{M}^-) \le 2d_2(\mathcal{M}^+ ,  m_n(0))=2d_2(\mathcal{M}^- ,  m_n(0)).
\] 
Let $\delta(q') = \inf_{ m\in\cPr,   d_2(m, \mathcal{M}^\pm)\ge q'} \cW_0(m)>0$, as in Lemma \ref{lemmaw0}.  Let  $0 \le s \le T$. Note that if  
$ d_2(m(t), \mathcal{M}^\pm)\ge q'$ for all $t \in [0,s]$,  then this implies 
\[
s \delta(q') \le \int_0^s \cW(m_n(t)) \, dt \le J_T(m_n, w_n) \le C'.
\]
Hence, for $T > \bar s := C'\big(\delta(q')\big)^{-1}$, by continuity of $m_n(t)$, since $ d_2(m_n(0), \mathcal{M}^\pm)>q'$,   there exists $s \in (0, \bar s)$   such that
\[d_2(m_n(t), \mathcal{M}^\pm)>q'
\quad  \text{for all $t \in [0, s)$  and }\quad d_2(m_n(s), \mathcal{M}^\pm)= q'.
\] Let $\bar m\in \mathcal{M}^+\cup \mathcal{M}^-$ such that  $d_2(m_n(s), \bar m)=q'$. We may assume without loss of generality that $\bar m\in \mathcal{M}^+$ (the proof is completely analogous if $\bar m\in \mathcal{M}^-$).  

So $d_2(m_n(s), \mathcal{M}^+)=d_2(m_n(s), \bar m)=q'$. 
Note that by symmetry of $m_n(t)$ we also have $d_2(m_n(T/2-s), \mathcal{M}^+)=d_2(m_n(T/2- s), \bar m) = q'$.
 Hence,  if $d_2(m_n(t), \mathcal{M}^+) > q$ for some $t \in (s, T/2- s)$,  by Lemma \ref{cruciallemma} it is possible to modify $(m_n, w_n)$ in $(s, T/2-s)$ to construct   a  competitor $(\mu_n, v_n)$  with  $J_T(\mu_n, v_n)<J_T(m_n, w_n)$.   Therefore, we can further restrict the minimization process to competitors $(m, w) \in \cK_T$ that   satisfy for some $s$ 
\begin{equation}\label{d2mpm}
\begin{cases}
d_2(m(t),  \mathcal{M}^+) < q & \forall t \in \left(s, \frac T2- s\right) \\
d_2(m(t),  \mathcal{M}^-) < q & \forall t \in \left(s-\frac T2, - s\right).
\end{cases}
\end{equation}
Note that $0<s\leq  \bar s = C'\big(\delta(q')\big)^{-1}$ and that $T> C'\big(\delta(q')\big)^{-1}\to +\infty$ as $q\to 0$. 

{\it Step 3: Existence of a minimizer.} By the growth condition \eqref{b}, we get that  there exists $t_n\in[0, T]$ such that 
$m_n(t_n)$ are  uniformly bounded in $\cP_2(\R^d)$ with respect to $n$.   Moreover by    \eqref{unmezzoh}
\begin{equation}\label{s1}
d^2_2(m_n(t), m_n(s)) \le C' |t-s|
\end{equation}
for all $t,s \in [0, T]$.   This implies that $(m_n)$ is uniformly continuous as a sequence of $\cP_2(\R^d)$-valued periodic functions, and
\begin{equation}\label{s2}
\sup_{n} \sup_{t \in [0,T]} \int_{\R^d} |x|^2 m_n(x,t) dx < \infty.
\end{equation} 
Therefore, by Ascoli-Arzel\`a theorem  and Lemma  \ref{equiconv}, $(m_n)$ has a  subsequence (still denoted by $(m_n)$) which converges in $C(\R, \cP_{p}(\R^d))$ for all $p<2$ 
 to some $m^T\in C_T(\R, \cP_p(\R^d))$.  Due to the lower semicontinuity  \eqref{lsc},  and the growth assumption \eqref{b} of $\cW$, we get that $m\in C_T(\R, \cP_2(\R^d))$.
Note that by convergence in $C(\R, \cP_{p}(\R^d))$ symmetry properties pass to the limit. Moreover also  \eqref{d2mpm} passes to the limit, due to lower semicontinuity of $d_2$ with respect to narrow convergence, see  Lemma \ref{equiconv}.

Finally, $(m_n)$ is bounded in $L^\infty(Q)$, so we can extract a further subsequence that converges $L^\infty(Q)$-weak-$^*$ to $m^T$, and $0 \le m^T(t,x) \le \rho$ a.e.  \\
Regarding $(w_n)$, we have
\begin{equation}\label{stimawn}
\int_0^T \int_{\R^d} |w_n(t,x)|^2 \, dx dt \le \rho \int_0^T \int_{\R^d} \frac{|w_n(t,x)|^2}{m_n(t,x)} \, dx dt,
\end{equation} 
hence $w_n$ converges weakly (up to a subsequence) in $L^2(Q)$ to some $w^T$.
\\It is easy to check that $-\partial_t m^T + {\rm div}(w^T) = 0$ in the distributional sense.\\
So we are just left to check that $(m^T, w^T)$ minimizes $J_T$. We use the lower semicontinuity of the kinetic part of the energy recalled in Proposition \ref{lsck},  and  for the potential part, we use the lower semicontinuity \eqref{lsc} of $\cW$ and Fatou lemma. 
\end{proof}

We now observe that $(m^T, w^T)\in\cK_T^{\rho, S}$ obtained in Theorem \ref{teo2} is a global minimizer of the control problem in Definition \ref{bo}. Note that, following  \cite{cas16}, this property can be used as a starting point to derive first order optimality conditions, that are of the form \eqref{mfg}. We mention that additional ``pressure'' terms and an ergodic constant may appear in the Hamilton-Jacobi equation, due to density constraints and $T$-periodicity. In any case, the key observation here is that no further multipliers related to $m(T/4 + t) = m(T/4  -t)$, $m(-t) = \gamma_\# m(t)$ are going to appear in the optimality conditions (by the symmetry assumptions on $\cW$).

\begin{corollary}\label{remsimm} $(m^T, w^T)\in\cK_T^{\rho, S}$ obtained in Theorem \ref{teo2} is a brake orbit in the sense of Definition \ref{bo}.
\end{corollary} 
\begin{proof} Having denoted by $f$ the derivative $f(x,m)= \frac{\delta}{\delta m} \cW (m) \in C(\R^d \times \cP(\R^d))$, given any minimizer $(\bar m, \bar w)$ of $J_T$ in $\cK_T^{\rho, S}$, it is possible to show by convexity of $(m,w)\mapsto  \frac{|w|^2}m$ and arguing as in \cite{BC16} that for all $(m, w) \in \cK_T^{\rho, S}$,
\begin{multline}\label{nasheq}
\int_0^T \int_{\R^d}  \left| \frac{\bar w(t,x)}{\bar m(t,x)}\right|^2 \bar m(t,x) + f(x,\bar m(t)) \bar m(t,x) dx dt \le \\ \int_0^T \int_{\R^d}  \left| \frac{w(t,x)}{m(t,x)}\right|^2 m(t,x) + f(x,\bar m(t)) m(t,x) dx dt.
\end{multline}
%Such a minimality of $(\bar m, \bar w)$ can be regarded as mean field Nash equilibrium property. 
Hence, we just need to show that the minimization property \eqref{nasheq} can be extended to the more general class of non-symmetric competitors $(m, w) \in \cK_T^{\rho} \supset \cK_T^{\rho,S}$. 

%To obtain \eqref{nasheq} for all $(m, w) \in \cK_T^{\rho} \supset \cK_T^{\rho, S}$, we notice first that if $(\bar m, \bar w)$ is a minimizer of $J_T$ and $\bar m(t,x) = \bar m(-t, \gamma(x))$, then $\bar w(t,x) = \bar w(-t, \gamma(x))$ (and the same holds regarding symmetry with respect to $T/4$). Indeed, since $\partial_t \bar m = {\rm div} \bar w$, $\bar w(t,x) = \bar w(-t, \gamma(x)) + \Phi$, where $\Phi$ is a divergence-free time dependent vector field. Therefore, $J_T(\bar m,\bar w) = J_T(\bar m,\bar w + \gamma\Phi(-t))$, which means that $\Phi \equiv 0$ on the support of $\bar m$ by strict convexity of $J_T(\bar m,\cdot)$ on $\bar m > 0$. On the other hand, $\bar w \equiv 0$ on the zero locus of $m$, hence if $\bar m(t,x)=\bar m(-t, \gamma( x)) = 0$, then $w(t, x)=0=w(-t, \gamma(x))$, yielding $\Phi = 0$.  

We detail only the fact that the symmetry condition $m(-t)=\gamma_\# m(t)$ can be dropped (arguing analogously, it can be shown the symmetry constraint around $T/4$ can be also dropped). Indeed, for $(m, w) \in \cK_T^{\rho}$ satisfying $m(T/4 + t) = m(T/4  -t)$ only, let
\[
 \tilde m (t) = \frac12 m(t) + \frac12 \gamma m(-t), \qquad \tilde w (t) = \frac12 w(t) + \frac12 \gamma w(-t).
\]
Note that $\cW(m) = \cW(\gamma_\# m)$ yields $f(x,m) = f(\gamma(x), \gamma m)$ (recall that $f = \delta_m \cW$), and therefore, since $\bar m(t,x) = \bar m(-t, \gamma(x))$  via a change of variables and convexity,
\begin{multline*}
\int_0^T \int_{\R^d}  \frac{|w(t,x)|^2}{m(t,x)} + f(x,\bar m(t)) m(t,x) dx dt =
\frac 12 \int_0^T \int_{\R^d}  \frac{|w(t,x)|^2}{m(t,x)} + f(x,\bar m(t)) m(t,x) dx dt + \\ + \frac 12 \int_0^T \int_{\R^d}  \frac{|w(-t,\gamma(x))|^2}{m(-t,\gamma(x))} + f(\gamma(x),\bar m(-t)) m(-t,\gamma(x)) dx dt \\ \ge
 \int_0^T \int_{\R^d}  \frac{|\tilde w(t,x)|^2}{\tilde m(t,x)} + f(x,\bar m(t)) \tilde m(t,x) dx dt.
\end{multline*}
Then, since we have that $(\tilde m , \tilde w) \in \cK_T^{\rho, S}$, 
\[
\int_0^T \int_{\R^d}  \frac{|w(t,x)|^2}{m(t,x)} + f(x,\bar m(t)) m(t,x) dx dt \ge \int_0^T \int_{\R^d}  \frac{|\bar w(t,x)|^2}{\bar m(t,x)} + f(x,\bar m(t)) \bar m(t,x) dx dt.
\]

\end{proof}

\section{Heteroclinic connections} \label{sechet} 

In this section we provide the proof of the second main result, that is Theorem \ref{teo3}. 

We introduce our definition of  heteroclinic connection. First of all we recall the definition of the energy  on  the whole space:
\begin{equation}\label{energy}
J(m,w)=\int_{-\infty}^{+\infty} \int_{\R^d} \frac12\left|\frac{dw}{dt\otimes m(t,dx)}\right|^2 m(t,dx)dt +\int_{-\infty}^{+\infty}\cW(m(t)) \, dt. 
\end{equation}
 Recall that $f(x,m)= \frac{\delta}{\delta m} \cW (m)$, and couples $(m,w)\in \mathcal{K}^\rho$ are admissible flows that are absolutely continuous with respect to the Lebsegue measure $dt \otimes dx$.
\begin{definition}\label{hetero} Let $ (\bar m, \bar w) \in \mathcal{K}^\rho$. We say that $(\bar m, \bar w)$ is a {\bf heteroclinic connection} for the MFG if 
 $\lim_{t\to -\infty} d_2(m(t), \mathcal{M}^-)=0=\lim_{t\to +\infty} d_2(m(t), \mathcal{M}^+)$, and $(\bar m, \bar w)$ satisfies
\begin{multline}\label{nequi}
\int_{-\infty}^{+\infty} \int_{\R^d}  \left| \frac{\bar w(t,x)}{\bar m(t,x)}\right|^2 \bar m(t,x) + f(x,\bar m(t)) \bar m(t,x) dx dt \le \\ \int_{-\infty}^{+\infty} \int_{\R^d}  \left| \frac{w(t,x)}{m(t,x)}\right|^2 m(t,x) + f(x,\bar m(t)) m(t,x) dx dt \qquad \forall (m,w)\in \mathcal{K}^\rho.
\end{multline}
\end{definition}

We start  observing that if $(m,w)$ has bounded energy, then $m$ should approach at $\pm \infty$ the stationary sets $\mathcal{M}^\pm$. 
\begin{lemma}\label{lemmat} 
Let $(m,w)\in \cK^\rho$,  and suppose $J(m,w)<+\infty$. Then \[
\text{ either } \lim_{t\to+ \infty}d_2\left( m(t), \mathcal{M}^+\right)=0 \quad \text{or } \quad  \lim_{t\to+ \infty}d_2\left( m(t), \mathcal{M}^-\right)=0\] and analogously for $t\to -\infty$. In particular if $(m,w)\in \cK^{\rho, S}$ then 
either \[
 \lim_{t\to+ \infty}d_2\left( m(t), \mathcal{M}^+\right)=0 \quad \text{and } \quad  \lim_{t\to- \infty}d_2\left( m(t), \mathcal{M}^-\right)=0\]
 or 
\[
 \lim_{t\to+ \infty}d_2\left( m(t), \mathcal{M}^-\right)=0 \quad \text{and } \quad  \lim_{t\to- \infty}d_2\left( m(t), \mathcal{M}^+\right)=0.\]
 \end{lemma}
 \begin{proof} Setting $C=J(m,w)$, observe that by    \eqref{unmezzoh} 
\begin{equation}\label{eq} d_2^2( m(t),  m(s))\leq C |t-s| \qquad \forall t,s\in \R.\end{equation}

So, due to the uniform continuity given by  \eqref{eq} and   to the fact that $J(m,w) = C<+\infty$, it is not possible that along some subsequence $t_n\to +\infty$, there holds 
that  $d_2(m(t_n), \mathcal{M}^\pm)\geq r$ for some $r>0$ independent of $n$. 
Indeed, using \eqref{eq}, we get that there exists $\eta=\eta(r)>0$ independent of $n$ such that $d_2(m(t), \mathcal{M}^\pm) \geq r/2>0$ for all $t\in [t_n-\eta, t_n+\eta]$.  By Lemma \ref{lemmaw0}, $\inf_{\{m\in\cPr,\  d_2(m, \mathcal{M}^\pm)\geq r/2\}} \cW(m)=\delta(r/2)>0$ and then we would get a contradiction with the fact that $J(m,w)=C<+\infty$ since 
\[C\geq \sum_{n} \int_{t_n-\eta}^{t_n+\eta}\cW(m)dt\geq\sum_{n} \int_{t_n-\eta}^{t_n+\eta}\delta(r/2)dt\to +\infty. \]

Therefore, using this observation, first of all we deduce  that for any sequence $t_n\to +\infty$,  either $\lim_n d_2(m(t_n), \mathcal{M}^+)=0$ or $\lim_n d_2(m(t_n), \mathcal{M}^-)=0$. 
Indeed, if it were not the case, 
then there would exist $\eps>0$  and a subsequence $t_n$ such that both $d_2(m(t_n), \mathcal{M}^+)\geq 2\eps$ and  $d_2(m(t_n), \mathcal{M}^-)\geq 2\eps>0$, contradicting the previous assertion.

Assume now that there are   sequences $t_n,s_n\to +\infty$ for which  $\lim_n d_2(m(t_n),\mathcal{M}^+)=0$ and $\lim_n d_2(m(s_n), \mathcal{M}^-)=0$. We may assume that
 $ s_n\leq t_n-1\leq t_n$ for all $n$.  
  
Fix $\delta\in(0,1)$ such that $2\delta<q_0 = \frac12 d_2(\mathcal{M}^+, \mathcal{M}^-)$  and let  $n_0$ such that for all $n\geq n_0$ \[d_2(m(t_n), \mathcal{M}^+)\leq \delta<\frac{q_0}{2} \qquad\text{and}\quad  d_2(m(s_n), \mathcal{M}^-)\leq \delta <\frac{q_0}{2}.\] 
Note that by triangular inequality   \[d_2(m(s_n), \mathcal{M}^+)\geq d_2(\mathcal{M}^+, \mathcal{M}^-)-d_2(m(s_n),\mathcal{M}^-)\geq 2 q_0-\delta\geq \frac32 q_0. \]
 The function $t\in (s_n,   t_n)\to m(t)$ is a continuous function with value in $\cP_2(\R^d)$. Therefore there exists $\bar t_n\in (s_n,t_n)$ such that \[d_2(m(\bar t_n),\mathcal{M}^+)= q_0.\] Again by triangular inequality we get that  \[d_2(m(\bar t_n), \mathcal{M}^-)\geq d_2(\mathcal{M}^+,\mathcal{M}^-)-d_2(m(\bar t_n), \mathcal{M}^+)\geq 2q_0-q_0=q_0.\] 
And this, again, would contradict the boundedness  of the energy. 

Therefore, we get that either for all $t_n\to +\infty$, $\lim_{n} d_2(m(t_n), \mathcal{M}^+)=0$, or  for all $t_n\to +\infty$,  $\lim_{n} d_2(m(t_n), \mathcal{M}^-)=0$. This implies  in particular  the conclusion, for $t\to +\infty$. Proceeding analogously we get the statement for $t\to -\infty$. Finally, if $(m,w)\in   \cK^{\rho, S}$ we conclude by recalling that $m(-t) = \gamma_\# m(t)$ and that $\mathcal{M}^+ = \gamma_\# \mathcal{M}^-$. 
 \end{proof}

 We provide now the existence of a solution to the constrained minimization problem
\begin{equation}\label{mininfty} J(\bar m,\bar w)=\min_{(m, w)\in \cK^{\rho, S}} J(m, w).\end{equation}

\begin{proof}[Proof of Theorem \ref{teo3}] \ \ 

\noindent {\bf Proof of item a)}. 
We use similar arguments to those in the proof of Theorem \ref{teo2}.
 
 {\it Step 1: energy bounds.} 
First of all we show that $\cK^{\rho, S}\neq\emptyset$. Choose   $m_0 \in \cPr$ with  compact support   such that $m_0 = \gamma m_0$ and let  \[d : = d_2(m_0, \mathcal{M}^+)=d_2(m_0,  \mathcal{M}^-).\]  Let $\bar m_+\in  \mathcal{M}^+$ such that  $d_2(m_0, \bar m_+)=d$. 
By Lemma \ref{connettori}, there exists a couple $(m, w) \in \cK^\rho$  that connects $m_0$ at time  $t = 0$ to
$\bar m_+$ at time $t = 1$.    
\[
\tilde m(t) := 
\begin{cases}
m(t) & t \in [0, 1], \\
\bar m_+ & t \in [1, +\infty), 
\end{cases}
\qquad
\tilde w(t,x) := 
\begin{cases}
w(t,x) & t \in [0, 1], \\
0 & t \in [1, +\infty). 
\end{cases}
\] Observe that $d_2(\tilde m(t),  \mathcal{M}^+)\leq d$ for all $t\in [0, 1]$. 
We extend   $(\tilde m, \tilde w)$  on $(-\infty, 0)$ symmetrically: \[(\tilde m(t), \tilde w(t)):=(\gamma \tilde m(-t), -\gamma \tilde w(-t)),\]
 Note that
 \[J(\tilde m,\tilde w)=2\int_0^1 \int_{\R^d} \left| \frac{\tilde w(t,x)}{\tilde m(t,x)}\right|^2 \tilde m(t,x) \, dx dt  +2 \int_0^1 \cW(\tilde m(t)) \, dt \\ \le  \frac{C'}{2},
\]  where  $C'$ is defined in \eqref{defc}. 

\smallskip
 
 {\it Step 2: limit of minimizing sequences}. 
We consider now a minimizing sequence $(m_n,w_n)\in \cK^{\rho, S}$ such that $J(m_n, w_n)\leq \frac{C'}{2}$.   
By the growth condition \eqref{b} on $\cW$, since $\int_0^1 \cW(m_n)dt<C'/2$, there exists $t_n\in[0,1]$ such that $m_n(t_n)$ is uniformly bounded in $\cP_2(\R^d)$. By \eqref{unmezzoh}, we get that $(m_n) \subset C(\R, \cP_2(\R^d))$ is equicontinuous and then  $m_n(t)$  is uniformly bounded in $\cP_2(\R^d)$ for all $t$.  By Ascoli-Arzel\`a theorem and Lemma \ref{lemmaconv},  up to extracting a  subsequence and to a diagonalization procedure, we get that $m_n$ converges  uniformly  
in $C([-L,L], \cP_{p}(\R^d))$ for all $p<2$ and all $L>0$,   to some $\bar m \in C(\R, \cP_p(\R^d))$. Again by lower semicontinuity \eqref{lsc},   and the  growth  condition \eqref{b}, there holds that  $\bar m \in C(\R, \cP_2(\R^d))$.  Moreover $\bar m(-t)=\gamma \bar m(t)$ since symmetry properties pass to the limit, and we can extract a further subsequence that converges also in   $L^\infty([-L,L]\times \R^d)$-weak-$^*$ to $\bar m$, so $0 \le \bar m(x,t) \le \rho$ a.e.. Finally,  reasoning as in \eqref{stimawn}, we get that $w_n$ converges   weakly (up to the extraction of a  subsequence and a diagonalization procedure) in $L^2([-L,L]\times \R^d)$ to some $\bar w$ for every $L>0$. In particular we get that $-\partial_t \bar m+{\rm div}(\bar w) =0$ in distributional sense in $(-\infty, +\infty)\times \R^d$. 

\smallskip
 
 {\it Step 3: finite energy.} 
We fix $L>0$.  By the lower semicontinuity properties and Fatou lemma, we get that  for every $L>0$, 
\[0\leq \int_{-L}^L   \int_{\R^d} \left| \frac{\bar w(t,x)}{\bar m(t,x)}\right|^2 \bar m(t,x) \, dx dt  + \int_{-L}^L \cW(\bar m(t)) \, dt\leq  \liminf_n J(m_n, w_n)=:\eta\leq  \frac{C'}{2}\] and so again by Fatou lemma
\begin{equation}\label{finiteenergy}0\leq J(\bar m,\bar w) \leq  \eta=\inf_{(\mu, v)\in \cK^{\rho, S}} J(\mu,v)\leq \frac{C'}{2}.\end{equation} 
This implies that $(\bar m, \bar w)\in \cK^{\rho,S}$ and moreover that $(\bar m, \bar w)$ is a minimizer.  

\smallskip

\noindent {\bf Proof of item b)}. 

\smallskip

 {\it Step 4: limit of $m^T(\pm T/4)$ as $T\to +\infty$.}  
First of all,  by Theorem \ref{teo2} observe that for all $\eps>0$ small and $T>\max(4s, \bar T)$ (where $s=s(\eps)$ and $\bar T=\bar T(\eps)$ according to the notation of  Theorem \ref{teo2}, in which $q=\eps$),  there exists  a minimizer $(m^T, w^T)$ such that 
\[ 0\leq d_2^2\left( m^T\left(t \right), \mathcal{M}^+\right) \leq \eps\qquad \forall t\in \left(s , \frac{T}{2}-s \right).\]
Since $\frac{T}{4}\in \left(s , \frac{T}{2}-s \right)$, sending $T\to +\infty$, we get that 
$\lim_{T\to +\infty} d_2^2\left( m^T\left(T/4 \right), \mathcal{M}^+\right) =0$ and, by  the symmetry conditions,   $\lim_{T\to +\infty} d_2^2\left( m^T\left(-T/4 \right), \mathcal{M}^-\right)=0$.

\smallskip
 
 {\it Step 5: equicontinuity of $m^T$ and passage to the limit.} 
Let fix $q\in (0, q_0)$ and let $(m^T, w^T)\in \cK_T^{\rho, S}$ be a minimizer of $ J_T$ constructed in Theorem \ref{teo2} with  $T>\bar T(q)(\geq 4)$.    First of all observe that by \eqref{enbound} there exists $C'$ independent of $T$ such that $0\leq J_T(m^T, w^T)\leq C'$, and so in particular by \eqref{unmezzoh}, we get that $(m^T(\cdot))_T\subset C(\R, \cP_2(\R^d))$ is equicontinuous.  By the growth condition \eqref{b}, since $\int_0^1\cW(m^T)dt\leq C'$, there exists  $t(T)\in [0, 1]$ such that  $m^T(t(T))$ is   bounded in $\cP_2(\R^d)$, uniformly with respect to $T$. 

By \eqref{unmezzoh} and triangular inequality we conclude that for all $t\in [0,T]$,  $m^T(t)$ is   bounded in $\cP_2(\R^d)$, uniformly with respect to $T$.  By Ascoli-Arzel\`a theorem and Lemma \ref{lemmaconv}, we get that up to extracting a  subsequence $T_n\to +\infty$ and using a diagonalization procedure, we get that $m^{T_n} $ converges  uniformly in $C([-L,L], \cP_{p}(\R^d))$ for all $p<2$ and all $L>0$, to some $m \in C(\R, \cP_p(\R^d))$, which a posteriori, due to \eqref{lsc} and \eqref{b}, satisfies $m \in C(\R, \cP_2(\R^d))$. 
  Moreover $m(-t)=\gamma m(t)$ since symmetry properties pass to the limit, and we can extract a further subsequence that converges also in   $L^\infty([-L,L]\times \R^d)$-weak-$^*$ to $m$, and $0 \le m(x,t) \le \rho$ a.e.. 
Finally,  reasoning as in \eqref{stimawn}, we get that $w_n$ converges   weakly (up to the extraction of a  subsequence and a diagonalization procedure) in $L^2([-L,L]\times \R^d)$ to some $w$ for every $L>0$. In particular we get that $-\partial_t m+{\rm div}(w) =0$ in distributional sense in $(-\infty, +\infty)\times \R^d$. 

It is immediate to check that the same argument applies to every limit point of  $(m^T,w^T)$. 
 
 \smallskip
 
 {\it Step 6: finite energy of   $(m,w)$.}  Let $(m,w)$ the limit of $(m^{T_n}, w^{T_n})$ as obtained in the previous step. Fix now $L>0$  and let  $n_0$ such that $T_n\geq 4L$ for all $n\geq n_0$.   By the lower semicontinuity properties and Fatou lemma, we get that  for every $L>0$, 
 we get that   
\begin{multline*} 0\leq \int_{-L}^L   \int_{\R^d} \left| \frac{ w(t,x)}{ m(t,x)}\right|^2   m(t,x) \, dx dt  + \int_{-L}^L  \cW( m(t)) \, dt\\
\leq  \liminf_n  \int_{-L}^L \int_{\R^d} \left| \frac{ w^{T_n}(t,x)}{ m^{T_n}(t,x)}\right|^2   m^{T_n}(t,x) \, dx dt  + \int_{-L}^L \liminf_n  \cW( m^{T_n}(t)) \, dt \\
\leq  \liminf_n  \int_{-L}^L \int_{\R^d} \left| \frac{ w^{T_n}(t,x)}{ m^{T_n}(t,x)}\right|^2   m^{T_n}(t,x) \, dx dt  + \liminf_n \int_{-L}^L   \cW( m^{T_n}(t)) \, dt \\
\leq \liminf_n  \left[\int_{-T_n/4}^{T_n/4}  \int_{\R^d} \left| \frac{ w^{T_n}(t,x)}{ m^{T_n}(t,x)}\right|^2  m^{T_n}(t,x) \, dx dt  + \int_{-T_n/4}^{T_n/4} \cW( m^{T_n}(t)) \, dt\right] \\
=\frac{1}{2} \liminf_n  J_{T_n}(m^{T_n}, w^{T_n})\leq \frac{1}{2} C'\end{multline*}  and so by Fatou lemma
\begin{equation}\label{finiteenergy2}0\leq J(m,w) \leq \frac{1}{2} C'.\end{equation} 
This, along with the properties   proved in Step 5,  implies that $(m,w)\in \cK^{\rho,S}$.

\smallskip

 {\it Step 7: $(m,w)$ is a solution of \eqref{mininfty}.} 
Since $J_{T_n}(m^{T_n}, w^{T_n})$ is equibounded, up to passing to a further subsequence we may assume that $\lim_n J_{T_n}(m^{T_n}, w^{T_n})=e>0$. 
Arguing as above it is immediate to check that
\[e=\lim_n J_{T_n}(m^{T_n}, w^{T_n})\geq 2J(m,w).\]

We claim that   \[e\leq 2J(\bar m,\bar w)\] where $(\bar m, \bar w)$ is a minimizer constructed in item a). If the claim is true, 
then we have that $2J(m,w)\leq e\leq 2J(\bar m, \bar w)$ which implies immediately that $(m,w)$ is a minimizer and moreover that $e=2J(\bar m,\bar w)$. 

Assume by contradiction that for some $\delta>0$, there holds 
 \[e>2J(\bar m,\bar w)+\delta.\]
Let us fix $T_n$ and consider $(\bar m,\bar w)$ restricted to $\left[-\frac{T_n}{4}, \frac{T_n}{4}\right]$. Extend them to $\left[-\frac{T_n}{4}, \frac{3T_n}{4}\right]$ by putting $\tilde m_n\left(t+\frac{T_n}{4}\right):=\bar m\left(\frac{T_n}{4}-t\right)$, $\tilde w_n\left(t+\frac{T_n}{4}\right):=-\bar w\left(\frac{T_n}{4}-t\right)$ for $t\in \left[0,\frac{T_n}{2}\right]$ and then extend them periodically in $\R$. It is easy to check that  $\gamma \tilde m_n(t)=\tilde m_n(-t)$. So $(\tilde m_n, \tilde w_n)\in \cK_{T_n}$ and  therefore \[J_{T_n}(m^{T_n}, w^{T_n})\leq  J_{T_n}(\tilde m_n, \tilde w_n)=2\int_{-\frac{T_n}{4}}^{\frac{T_n}{4}} \int_{\R^d} \left| \frac{\bar w(t,x)}{\bar m(t,x)}\right|^2 \bar m(t,x) \, dx  +\cW(\bar m(t))dt\leq 2J(\bar m,\bar w)<e-\delta.\] Taking $n$ sufficiently large, this gives a contradiction with the fact that $e=\lim_n J_{T_n}(m^{T_n}, w^{T_n})$. 

\end{proof} 
Finally arguing as in Corollary \ref{remsimm}, one can show that any constrained minimizer  in $\cK^{\rho, S}$  (so under a symmetry constraint with respect to $t = 0$)   is an heteroclinic connection, in the sense of Definition \ref{hetero}. In other words, it has the minimality property \eqref{nequi} in the broader class of competitors $\cK^{\rho}$.
 
\begin{corollary}\label{ultimoc}  Let $(m, w)\in\cK^{\rho, S}$ be any solution to the minimization problem \eqref{mininfty}. Then $(m,w)$ is an heteroclinic connection in the sense of Definition \ref{hetero}.
\end{corollary} 
\begin{proof} We omit the proof, since it  follows the same lines as the proof of Corollary \ref{remsimm}. \end{proof} 

\section{A model problem} \label{secmodel}
In this section we describe a model to which results of the previous sections apply. 
We define on $\cPr$ (for any fixed $\rho>0$) the following potential energy
 \begin{equation}
\label{sta}\mathcal{W}(m)=\int_{\R^d} W(x)m(dx)-\int_{\R^d}\int_{\R^d} K(|x-y|)m(dx)m(dy).
\end{equation}
First of all we describe our main assumptions on $K$ and $W$ and then we check  all the conditions that are needed in Theorems \ref{teo2}, \ref{teo3}. Note that, as we will see below, $\mathcal{W}(m)$ has minimizers on $\cPr$, but $\min_\cPr \cW < 0$. Therefore, to apply Theorems \ref{teo2}, \ref{teo3} one just needs to add to $\cW$ the renormalization constant $\min_\cPr \cW$, that is to consider
\[
\cW_0(m) = \cW(m) - \min_\cPr \cW.
\]
\subsection{Standing assumptions on $W$ and $K$}\label{asssec}
 We start describing the assumptions on the local energy $\int_{\R^d}W(x)dm(x)$. Let $W:\R^d\to[0, +\infty)$ be a  confining double-well potential such that 
\begin{equation}\label{assw} \begin{split} & \bullet\,\text{$W\in C(\R^d)$ is non-negative},\ \\ & \bullet\,  \exists C>0,\  \text{  such that  }C^{-1}|x|^2-C\leq W(x)\leq C|x|^2+C\\ 
& \bullet\, \exists a^+, a^-\in \R^d, \tilde r > 0 \text{ such that  }B(a^+, \tilde r) \cap B(a^-, \tilde r) = \emptyset \\ & \quad
  \qquad { and } \quad W(x)=0 \Leftrightarrow x\in B(a^+, \tilde r)\cup B(a^-, \tilde r).\end{split} 
\end{equation} 
Note that we require $W$ to have two disjoint flat regions $B(a^\pm, \tilde r)$. Moreover, we assume that $W$ is invariant under a  reflection $\gamma:\R^d\to \R^d$, that is
\begin{equation}\label{ref} W(x)= W(\gamma(x))\qquad x\in\R^d.
\end{equation}  In particular this implies that $\gamma (a^+)=a^-$ and $|a^\pm| - \tilde r > 0$. 
Finally, we assume  that  the plateaus of $W$ are sufficiently large with respect to the density constraint $\rho$, in the following sense:
 \begin{equation}\label{ass3} \rho\geq \tilde \rho :=\frac1{\omega_d \tilde r^d},  
\end{equation} where $\tilde r$ is defined  in \eqref{assw}, and $\omega_d$ is the Lebesgue measure of $B_1$.  See Figure \ref{fig2} for an example of $W$ satisfying our assumptions.

 We describe now the assumptions on the interaction energy $-\int_{\R^d}\int_{\R^d} K(|x-y|)m(dx)m(dy)$ and some basic properties. We consider a radially symmetric  interaction kernel  $K(|x|)$, where $K:[0,+\infty)\to [0, \infty)$ is a function such that  
\begin{equation}\label{assK}\begin{split} & \bullet\, r\mapsto  r^{d-1} K(r)  \in L^1_{loc}([0, +\infty), [0, +\infty)), \\ & \bullet\, \text{$K$ is nonincreasing}, \\
& \bullet\, \text{$\lim_{r\to 0} K(r)-K(t+r)>0$ for every $t$}, \\ & \bullet\, \lim_{r\to +\infty} K(r)=0.\end{split}
\end{equation}
Moreover,  we assume that $K$ is positive definite, which means  that \begin{equation}\label{pos} \int_{\R^d}\int_{\R^d} f(x)f(y)K(|x-y|) dxdy\geq 0\quad \text{ for all $f\in L^1(\R^d)$}\end{equation}  and $\int_{\R^d}\int_{\R^d} f(x)f(y)K(|x-y|) dxdy= 0$ if and only if $f=0$. Note that positive definiteness is    equivalent to assume that the Fourier transform of $K$ is a positive function. 

We define the energy interaction functional for $f\in L^1(\R^d)$ 
\begin{equation}\label{int} \mathcal{I}(f)=\int_{\R^d}\int_{\R^d} f(x)f(y) K(|x-y|)dxdy=\int_{\R^d} f(x)V_f(x)dx
\end{equation} where $V_f$ is the interaction potential 
\begin{equation}\label{pot} V_f(x)=f*K(x)= \int_{\R^d}  f(y) K(|x-y|)dy.
\end{equation} 
\begin{remark}\upshape
If $K$ is positive definite, then the map $f\mapsto V_f=K*f$ is monotone increasing in the sense that
\[\int_{\R^d} (V_{f_1}-V_{f_2})(f_1-f_2)dx \geq 0\qquad \forall f_1, f_2\in L^1(\R^d). \]
\end{remark}
We recall a well known result on the interaction potential. 
\begin{lemma}\label{lemmapot} Assume \eqref{assK}. 
Let $f\in L^1(\R^d)\cap L^\infty(\R^d)$. Then $V_f\in C(\R^d)$ and $\lim_{|x|\to +\infty} V_f(x)=0$.
\end{lemma} 
\begin{proof} Let $\eta:[0, +\infty)\to [0,1]$ be a smooth function such that $\eta=0$ in $[0,1]$ and $\eta=1$ in $[2, +\infty)$ and define $K_\eps(|x|)=K(|x|)\eta (|x|/\eps)$ for $\eps>0$. 
Then $V^\eps_f=K_\eps*f\in C(\R^d)$ since $|V^\eps_f(x+h)-V^\eps_f(x)|\leq \|K_\eps\|_\infty\|f(\cdot +h)-f(\cdot)\|_1\to 0$ as $|h|\to 0$. 
Moreover $\lim_{|x|\to +\infty} V_f^\eps(x)=0$: indeed for $R>2$ and $\eps\leq 1$, recalling the assumptions on $K$, 
\begin{eqnarray*} |V_f^\eps(x)| &\leq&    \int_{B(0, R)} K_\eps(|y|)|f(x-y)|dy+  \int_{\R^d\setminus B(0, R)} K_\eps(|y|)|f(x-y)dy|\\&\leq& K(\eps)\int_{B(x,R)} |f(y)|dy+ K(R)\|f\|_1\to 0\qquad \text{ as $|x|\to +\infty$ and $R\to +\infty$}.\end{eqnarray*} 
We conclude observing that  $V^\eps_f$ converges uniformly to $V_f$ as $\eps\to 0$ since 
\[|V^\eps_f(x)-V_f(x)|\leq  \int_{B(0, 2\eps)} K(|y|) |f(x-y)|dy\leq \|f\|_\infty  \int_{B(0, 2\eps)} K(|y|)  dy.\]  
%
%Now we prove that $V$ is vanishing at $\infty$. Observe that for $R>1$, 
%\[V_f(x)= \int_{B(0, 1/R)} K(|y|)f(x-y)dy+ \int_{B(0, R)\setminus B(0, 1/R)} K(|y|)f(x-y)dy+  \int_{\R^d\setminus B(0, R)} K(|y|)f(x-y)dy.\]
%By our assumptions on $K$, we get that $ |\int_{\R^d\setminus B(0, R)} K(|y|)f(x-y)dy|\leq K(R)\|f\|_1\to 0$  and 
%$| \int_{B(0, 1/R)} K(|y|)f(x-y)dy|\leq \|f\|_\infty \int_{B(0, 1/R)} K(|y|)dy\to 0$ as $R\to +\infty$. Therefore for every $\eta>0$, there exists $R_\eta$, not depending on $x$, such that
%\[\left|V_f(x)-  \int_{B(0, R)\setminus B(0, 1/R)} K(|y|)f(x-y)dy\right|\leq \eta\qquad \forall \ R\geq R_\eta.  \]
%Finally,   we fix $R>R_\eta$,  and we observe that  for all $|x|>2R$, we get 
%\begin{multline*} \left|\int_{B(0, R)\setminus B(0, 1/R)} K(|y|)f(x-y)dy\right|=  \left|\int_{B(x, R)\setminus B(x, 1/R)} K(|y-x|)f(y)dy\right| \\ \leq  K(1/R)  \int_{B(x, R)}  |f(y)|dy\leq K(1/R)  \int_{\R^d\setminus B(0, |x|/2)}  |f(y)|dy
%\to 0 \qquad \text{as }|x|\to +\infty.   \end{multline*} So there exists $R'>2R $ such that $\left|\int_{B(0, R)\setminus B(0, 1/R)} K(|y|)f(x-y)dy\right|\leq \eta$ if $|x|>R'$. 
%
%In conclusion for all $\eta>0$ there exists $R'>0$ such that $|V_f(x)|\leq 2\eta$ for all $|x|>R'$. 
\end{proof} 
We recall  the Riesz rearrangement inequality (see \cite{LLbook}).  For $f\in L^1(\R^d)$ such that $f\geq 0$, we define the $f^*$ is the spherical rearrangement of $f$, that is 
\[f^*(x)=\int_0^{+\infty} \chi_{\{y\ | f(y)>t\}^*} (x)dt \quad \text{ where } \{y\ | f(y)>t\}^*=B(0, r), \text{ with }\omega_n r^n=|\{y\ | f(y)>t\}|.\]
The Riesz rearrangement inequality states that 
\[\int_{\R^d}\int_{\R^d} f(x)f(y) K(|x-y|)dxdy\leq \int_{\R^d}\int_{\R^d} f^*(x)f^*(y) K^*(|x-y|)dxdy.\] 
Note that $K^*(|x|)= K(|x|)$ since $K$ is radially symmetric and nonincreasing. So we conclude that 
for every $f\in L^1(\R^d)$ such that $f\geq 0$, 
\begin{equation}\label{riesz}  \mathcal{I}(f)\leq \mathcal{I}(f^*).
\end{equation}  
We recall a well known result, see \cite{LLbook}. 
%The quantitative version of the Riesz inequality has been obtained recently in \cite{fp}. 
\begin{lemma}\label{cKbounds} Assume \eqref{assK}.
Let $r_\rho=(1/\rho\omega_d)^{1/d}$. 
There holds
\[\sup_{m\in \cPr}\mathcal{I}(m)=\mathcal{I}(\rho\chi_{B_{r_\rho}})> 0.\]
%
%If moreover  $d\geq 2$,  $K(r)= r^{\alpha-d}$ for $1<\alpha<d$,  there exists a constant $C=C(d,\alpha)>0$ such that for all measurable sets $E\subseteq \R^d$ with $|E|=1/\rho$, there holds 
%\[ \mathcal{I}(\rho\chi_{B_{r_\rho}})-\mathcal{I} (\rho\chi_E) \geq \rho^{3-\alpha/d} \inf_{x\in \R^d} |E\Delta B(x, r_\rho)|^2. \]
\end{lemma}
\begin{proof} It follows from the Riesz rearrangement inequality that maximizers  are radially symmetric and nonincreasing. 
The fact that maximizers of $\mathcal{I}(f)$  in $ \cPr$ are characteristic functions it can be proven looking at the second variation of the functional (see \cite{LLbook}, and see also  the following Proposition \ref{propminima} for a similar argument).  \end{proof}
\begin{remark}\label{remb} \upshape
Note that, due to the fact that $W(x)\geq 0$ and to Lemma \ref{cKbounds}, we get that for all $m\in\cPr$ there holds
\[-\mathcal{I}(\rho\chi_{B_{r_\rho}})\leq\ \cW(m)\leq \int_{\R^d} W(x)m(x).\]
\end{remark} 
 \subsection{Assumptions \eqref{rif}, \eqref{b}} 
 We check that $\cW$ defined in \eqref{sta}, under the standing assumptions in Section \ref{asssec}, % is aggregating and moreover
  satisfies  the growth condition \eqref{b}, and the reflection invariance \eqref{rif}.
 \begin{proposition} Under the assumptions \eqref{assw}, \eqref{ref}, \eqref{assK},  \eqref{pos}, the functional $\cW$ in \eqref{sta} satisfies \eqref{rif}, \eqref{b}.
 \end{proposition} 
 Since $\cW$ and $\cW_0$ differ by a constant, the same conclusion holds for $\cW_0$. Moreover, note that $ \frac{\delta }{\delta m} \cW(m)=  W(x)-\int_{\R^d}  K(|x-y|)m(dy)$. Hence, as a direct consequence of the positivity of $K$ assumed in \eqref{pos}, $\cW$ (and $\cW_0$) is ``aggregating'', namely it satisfies $\int_{\R^d}( f(x,m)- f(x,m'))d(m-m')(x) \le 0$ for all $m, m'$.
  
 \begin{proof}  
 We observe that by \eqref{ref} and the symmetry properties of $K$, 
 \begin{multline*} \cW(\gamma(m))=\int_{\R^d} W(x)\gamma m(dx)-\int_{\R^d}\int_{\R^d} K(|x-y|)\gamma m(dx)\gamma m(dy)\\=\int_{\R^d} W(\gamma x)  m(dx)-\int_{\R^d}\int_{\R^d} K(|\gamma x-\gamma y|)m(dx) m(dy) =\cW(m)\end{multline*} which is \eqref{rif}. 
 Finally, by \eqref{assw} and Lemma  \ref{cKbounds} we get that for all $m\in \cPr$, there holds 
 \[C^{-1}\int_{\R^d}|x|^2 m(dx)-C -\mathcal{I}(\rho\chi_{B_{r_\rho}}) \leq \cW(m)\leq  C\int_{\R^d}|x|^2 m(dx)+C, \]
% Note that for every compact set $\mathcal{M}\in \mathcal{P}_2(\R^d)$  and every $m\in \mathcal{K}_\rho$, using the definition of $d_2$ and triangle inequality we get, letting $\delta_0$ the Dirac measure in $0$, 
% \[\min_{n\in \mathcal{M}}d_2(m ,n)-\max_{n\in \mathcal{M}} d_2(\delta_0,n)\leq \int_{\R^d}|x|^2 m(dx)=d_2(m, \delta_0)\leq \min_{n\in \mathcal{M}}d_2(m ,n)+\max_{n\in \mathcal{M}} d_2(\delta_0,n).\] Substituting this estimate in the previous inequality and letting $\mathcal{M}=\mathcal{M}^+\cup\mathcal{M}^-$, we get 
 which is \eqref{b}.
 \end{proof}
 
 \subsection{Lower semicontinuity properties of $\cW$: assumption \eqref{lsc}}
We provide continuity and semicontinuity properties of $\cW$ (and $\cW_0$). Let us first check that $\cW$ is lower semicontinuous with respect narrow convergence, %, in $\cPr$. 
which implies \eqref{lsc}. 

\begin{proposition}\label{propcontK} Assume that $K$ satisfies \eqref{assK} and let $\mathcal{I}$ be as in \eqref{int}. Let $m_k, m \in \cPr$ and suppose that  $m_k$ converges to  $m$ narrowly. Then, $\lim_k \mathcal{I}(m_k)=\mathcal{I}(m)$. In particular, the functional $\cW$ satisfies \eqref{lsc}.
\end{proposition} 

%\begin{proposition}\label{propcontK} 
%The functional $\cW$ satisfies \eqref{lsc}. In particular  let  $m_k, m$ be Borel probability measures on $\R^d$ such that  $m_k\to m$ narrowly. Then the following holds. 
%
%\begin{enumerate}\item  
%\[\liminf_k \int_{\R^d} W(x)m_k(dx)\geq \int_{\R^d} W(x)m(dx)\] where $W$ satisfies \eqref{assw}. 
%
%\item If moreover    $m_k, m \in \cPr$,  then 
%\[\lim_k \mathcal{I}(m_k)=\mathcal{I}(m)\] where $\mathcal{I}$ has been introduced in \eqref{int} and $K$ satisfies \eqref{assK}. 
%\end{enumerate} 
%  
%\end{proposition} 

\begin{proof} Once $\lim_k \mathcal{I}(m_k)=\mathcal{I}(m)$ is established, the lower semicontinuity \eqref{lsc} follows by applying Lemma \ref{lemmaconv} (and i) $\Rightarrow$ ii) of Lemma \ref{equiconv}) to the term $\int W(x)m_k(dx)$. 
We sketch briefly  the proof of $\lim_k \mathcal{I}(m_k)=\mathcal{I}(m)$. Similar arguments have been used in \cite[Lemma 3.3]{cho}.

By Cauchy-Schwartz inequality, recalling the positive definiteness of $K$,  we get (note that we are identifying  $m_k, m$ with their densities) 
\[ \left|\int_{\R^d}\int_{\R^d}  m_k(x) m(y)K(|x-y|)dxdy\right|\leq \mathcal{I}(m_k)^{1/2}\mathcal{I}(m)^{1/2}\] from which we deduce 
\[ \int_{\R^d}\int_{\R^d} (m_k(x)-m(x))(m_k(y)-m(y))K(|x-y|)dxdy \geq 
 \left((\mathcal{I}(m_k)^{1/2}-(\mathcal{I}(m)^{1/2}\right)^2.\]
We fix $R>0$ and we write, recalling the conditions on $K$, 
\begin{multline}\label{uno} 
\int_{\R^d}\int_{\R^d} (m_k(x)-m(x))(m_k(y)-m(y))K(|x-y|)dxdy\\
%\leq \left|\int_{\R^d}\int_{\R^d} (m_k(x)-m(x))(m_k(y)-m(y))K(|x-y|)\chi_{|x-y|\leq R}dxdy\right|\\
%+
%\left|\int_{\R^d}\int_{\R^d} (m_k(x)-m(x))(m_k(y)-m(y))K(|x-y|)\chi_{|x-y|>R}dxdy\right|\\
\leq  \left|\int_{\R^d}\int_{\R^d} (m_k(x)-m(x))(m_k(y)-m(y))K(|x-y|)\chi_{|x-y|\leq R}dxdy\right|+4  K(R).
\end{multline} 
 We define
\[F_k(x)=\int_{\R^d} m_k(y)K(|x-y|)\chi_{|x-y|\leq R}dy,\quad \text{and}\quad F(x)=\int_{\R^d} m(y)K(|x-y|)\chi_{|x-y|\leq R}dy\] and we observe that $F_k, F\in L^1(\R^d)$.
Then 
\begin{multline}\label{due} \int_{\R^d}\int_{\R^d} (m_k(x)-m(x))(m_k(y)-m(y))K(|x-y|)\chi_{|x-y|\leq R}dxdy\\= -\int_{\R^d}\int_{\R^d} (m_k(x)-m(x)) F(x)dx +\int_{\R^d}\int_{\R^d} m_k(x)(F_k(x)- F(x))dx.  \end{multline} 
Observe that since $m_k\to m$ narrowly and $m_k, m\leq \rho$, then $m_k\to m$   weak* in $L^\infty$, due to density of continuous functions in $L^1$. 
Therefore  \begin{equation}\label{tre} \lim_k \int_{\R^d} (m_k(x)-m(x))F(x)dx=0.\end{equation} Moreover $\lim_k F_k(x)=F(x)$ for a.e. $x$ and
\[\|F_k\|_1= \|K(|x|)\chi_{|x|\leq R}\|_1\to \|F\|_1= \|K(|x|)\chi_{|x|\leq R}\|_1.\] 
Therefore by Fatou lemma, $F_k\to F$ in $L^1(\R^d)$, which implies, recalling that $m_k\leq\rho$ that 
\begin{equation}\label{quattro} \lim_k \int_{\R^d} (F_k(x)-F(x))m_k(x)dx=0.\end{equation}
So, using \eqref{tre}, \eqref{quattro} in \eqref{uno} and recalling that $K(R)\to 0$ as $R\to +\infty$, we get the conclusion. 
\end{proof} 
We observe  the following fact about uniformly integrability of narrowly convergent sequences of measures.
\begin{lemma}\label{remw}   Let $W$ is as in \eqref{assw} and  $\mu_k, \mu\in \mathcal{P}_2(\R^d)$ such that $\mu_k\to \mu$ narrowly. 
Then,  $\lim_k d_2(\mu_k,\mu)=0$ if and only if  $\lim_k\int_{\R^d} W(x) \mu_k(dx)= \int_{\R^d} W(x) \mu(dx)$. 
\end{lemma} 
\begin{proof}  
We observe that, by Lemma \ref{equiconv}, $\lim_k d_2(\mu_k, \mu)=0$ is equivalent to 
the fact that $\mu_k$  has uniformly integrable  $2$-moments, that is 
\begin{equation}\label{int1} \lim_{R\to +\infty} \sup_{k}\int_{\R^d\setminus B(0, R)} |x|^2d\mu_k(x)=0.\end{equation} 
 
 Let $R>0$, sufficiently large, such that $RC-C^2>1$ and $RC^{-1}>3$, where $C$ is the constant appearing in \eqref{assw}.  We denote $A_R:=\{W(x)\geq R\}$. Then by \eqref{assw} we get $\R^d\setminus B(0, \sqrt{RC^{-1}+1})\subseteq A_R\subseteq \R^d\setminus B(0, \sqrt{RC-C^2})$. Then, recalling \eqref{assw}, we get 
\begin{multline*}\frac{C}{2} \sup_{k} \int_{\R^d\setminus B(0, \sqrt{RC^{-1}+1})} |x|^2d\mu_k(x)
% \leq \sup_{k} \int_{\R^d\setminus B(0, \sqrt{RC^{-1}+1})} (C|x|^2-C) d\mu_k(x) \\ 
%\leq \sup_{k}\int_{\R^d\setminus B(0, \sqrt{RC^{-1}+1})} W(x) d\mu_k(x)
\leq \sup_{k}\int_{A_R} W(x) d\mu_k(x) \\
%\leq \sup_{k}\int_{\R^d\setminus B(0, \sqrt{RC-C^2})} W(x) d\mu_k(x)\\
%\leq   \sup_{k} \int_{\R^d\setminus B(0, \sqrt{RC-C^2})} (C^{-1} |x|^2+C) d\mu_k(x)
\leq(C^{-1} +C)  \sup_{k} \int_{\R^d\setminus B(0, \sqrt{RC-C^2})}   |x|^2 d\mu_k(x). \end{multline*} Sending $R\to +\infty$,  we get that $\mu_k$  has uniformly integrable  $2$-moments, that is \eqref{int1} holds, 
  if and only if 
 $W$ is uniformly integrable with respect to $\mu_k$, that is 
\begin{equation}\label{int2}\lim_{R\to +\infty} \sup_{k}\int_{\{W(x)\geq R\}} W(x) d\mu_k(x)=0.\end{equation}
 
 Now observe that if $\mu_k\to \mu$ narrowly, then $\lim_k \int_{\{W(x)\leq R\}} W(x) d\mu_k(x)= \int_{W(x)\leq R} d\mu(x)$ for all $R>0$. 
Therefore if $\mu_k\to\mu$ narrowly, then  $\lim_k \int_{\R^d} W(x) d\mu_k(x)= \int_{\R^d} d\mu(x)$ if and only if \eqref{int2} holds. This gives the conclusion. 
\end{proof}

\subsection{Minimizers for the stationary problem: assumptions \eqref{z} and   \eqref{c} } \label{secsta} 
We start proving  existence and qualitative properties  of minimizers of $ \mathcal{W}(m)$ in the set $\cPr$.  In particular we will show that under assumptions   \eqref{assw}  and \eqref{assK}  the minimizers are characteristic functions. Moreover we show that  assumption \eqref{ass3} assures that the minimizers are characteristic functions of balls. 
Then, in Proposition \ref{classiminimi} we show that \eqref{z} is satisfied by $\cW_0$  and  in Proposition \ref{continu}, we prove that $\mathcal{W}_0$ satisfies assumption \eqref{c}. 

The first result is about qualitative properties of minimizers, by looking at the first and second variation of the functional.
\begin{proposition}\label{propminima} Assume \eqref{assw}  and  \eqref{assK}. Let $m\in \cPr$ be a minimizer of the functional $\mathcal{W}$. Then 
 there exists a  bounded, measurable set $E\subseteq \R^d$ such that $m=\rho\chi_E$. 
\end{proposition} 
\begin{proof} The proof is based on analogous arguments as in \cite[Prop. 5.4, Thm 5.7]{cnov}.
Let  $m\in \cPr$ be a  minimizer of the functional $\mathcal{W}$, and let define the sets 
\[S=\{x\ | m(x)=\rho\},\qquad N=\{x\ | \ m(x)=0\},\] as subsets of the set of density points of $m$. 

We  start showing that $m$ has bounded support.
 We compute the first variation of the functional $\mathcal{W}$  as in \cite[Lemma 5.3]{cn18}.  
Pick   any $\psi, \phi\in L^1(\R^d, [0, \rho])$ such that $\int_{\R^d} \phi(x)dx=\int_{\R^d}\psi(x)dx$,
$\psi=0$  a.e. in $S$ and $\phi=0$ a.e. in $N$. Then  $\int_{\R^d} m+\lambda(\psi-\phi)dx=1$ for every $\lambda$ and moreover for $\lambda\in (0,1)$ we get that
 $0\leq m(x)+\lambda(\psi(x)-\phi(x))\leq \rho$ for a.e. $x\in N\cup S$. 
 
 Now we consider two sequences $\psi_\eps\to \psi$, $\phi_\eps\to \phi$ in $L^1$ such that $\int_{\R^d} \phi_\eps(x)dx=\int_{\R^d}\psi_\eps(x)dx=\int_{\R^d} \phi(x)dx$ and such that $\phi_\eps=0$ a.e. on the set $\{m(x)\leq \eps\}$ and $\psi_\eps=0$ a.e. on the set $\{m(x)\geq1-\eps\}$. Choosing $\lambda$ sufficiently small (depending on $\eps$) we get that $m+\lambda(\psi_\eps-\phi_\eps)\in \cPr$. 
Using the fact that $\mathcal{W}(m)\leq \mathcal{W}(m+\lambda(\psi_\eps-\phi_\eps))$, we get, sending $\lambda\to 0^+$, 
\[\int_{\R^d} (-2V_m(x)+W(x))(\psi_\eps(x)-\phi_\eps(x))dx\geq 0
\]  where $V_m$ is the potential of $m$ defined in \eqref{pot}. Sending $\eps\to 0$ we conclude that 
\begin{equation}\label{e1} \int_{\R^d} (-2V_m(x)+W(x))(\psi(x)-\phi(x))dx\geq 0.
\end{equation} 
If we choose $\psi,\phi$ such that $\psi=\phi=0$ in $N\cup S$, then we can exchange the role of $\phi, \psi$ in \eqref{e1} and obtain 
\[\int_{\R^d\setminus (S\cup N)} (-2V_m(x)+W(x))(\psi(x)-\phi(x))dx=0\] for all $ \phi-\psi\in L^1(\R^d)$,  such that $\phi-\psi=0$ in $S\cup N$, with  $\int_{\R^d}(\phi-\psi) dx=0$.
This implies by the fundamental lemma of calculus of variations,  that there exists a constant $c$ such that
\begin{equation}\label{e3} -2V_m(x)+W(x)=c\qquad x\in \R^d\setminus (N\cup S).\end{equation}
Using this fact in \eqref{e1} we get, taking $\psi=0$ in $S\cup N$, and observing that $\int_{\R^d\setminus (S\cup N)} \psi dx=\int_{\R^d} \phi dx$, 
\[0\leq  \int_{S} (2V_m(x)-W(x)) \phi(x)dx +c\int_{\R^d\setminus S\cup N} (\psi-\phi )dx=\int_{S} (2V_m(x)-W(x)+c) \phi(x)dx  \]
an analogously, taking $\phi=0$ in $S\cup N$,
\[0\leq  \int_{N} (-2V_m(x)+W(x)) \psi(x)dx +c\int_{\R^d\setminus S\cup N} (\psi-\phi )dx=\int_{N} (-2V_m(x)+W(x)-c) \psi dx. \]
This implies that 
\begin{equation}\label{e4} \begin{cases}-2V_m(x)+W(x)\geq c& x\in N\\ -2V_m(x)+W(x)\leq c& x\in S,\end{cases}
\end{equation} Recalling that $W$ is coercive (see assumption \eqref{assw}) and $V_m$ vanishes at infinity, see Lemma \ref{lemmapot}, we conclude  from \eqref{e3}, \eqref{e4} that both  $S$ and $\R^d\setminus (N\cup S)$ have to be bounded.  This implies that the support of $m$, which is given by $\R^d\setminus N $,  is bounded. 

Now we show that $m(x)\in \{0, \rho\}$ for a.e. $x$.  We compute the second variation of the functional as in \cite[Lemma 5.5]{cnov}.  We consider $\phi, \psi$ as above,  and we assume that $\psi=\phi=0$ in $N\cup S$. Now  we consider two sequences $\psi_\eps\to \psi$, $\phi_\eps\to \phi$ in $L^1$ such that $\int_{\R^d} \phi_\eps(x)dx=\int_{\R^d}\psi_\eps(x)dx=\int_{\R^d} \phi(x)dx$ and such that $\phi_\eps=0$ and $\psi_\eps=0$ a.e. on the set $\{m(x)\leq \eps\}\cup\{m(x)\geq1-\eps\}$.  Now, for $\lambda$ sufficiently small, depending on $\eps$, $m+\lambda (\psi_\eps-\phi_\eps)\in \cPr$. By minimality of $m$, recalling that $W(x)-2V_m(x)$ is constant on $\R^d\setminus (N\cup S)$, that  $\psi_\eps-\phi_\eps=0$ in $N\cup S$ and that  $\int_{\R^d} (\psi_\eps(x)-\phi_\eps(x))dx=0$, we get 
\[  - \lambda^2\int_{\R^d}K(|x-y|)(\psi_\eps(x)-\phi_\eps(x))(\psi_\eps(y)-\phi_\eps(y))dxdy\geq 0.\]
Sending $\eps\to 0$ we get 
\begin{equation}\label{e5} \int_{\R^d\setminus (N\cup S)}\int_{\R^d\setminus (N\cup S)} (\psi(x)-\phi(x))(\psi(y)-\phi(y))K(|x-y|)dxdy\leq 0.\end{equation}
Assume now by contradiction that there are two Lebesgue  points $x,y$ of $m$ such that $0<m(x), m(y)<\rho$.  Let $d=|x-y|$.  
Since $x,y$ are density points for $m$, it is possible to find, for   $0< \eps<d/4$ sufficiently small,   two sets with positive measure $A_\eps(x)$, $A_\eps(y )$ such that 
\begin{itemize}\item  $A_\eps(x)\subseteq  B(x,\eps)\cap \R^d\setminus (N\cup S)$, \item  $A_\eps(y)\subseteq  B(y,\eps)\cap \R^d\setminus (N\cup S)$, \item  $d(A_\eps(x), A_\eps(y))\geq d/2$,  \item  $|A_\eps(x)|=|A_\eps(y)|>0$.  \end{itemize} We recall that $ B(x,\eps)$ is the ball centered at $x$ of radius $\eps$, and analogously $B(y,\eps)$. 
Observe that  if  either $t,z \in A_\eps(x )$ or  $t,z \in A_\eps(y)\subseteq B(y,\eps)$, then $|t-z|\leq 2\eps$, and if $t\in A_\eps(x)$, $z\in A_\eps(y)$,  then $|t-z|\geq d(A_\eps(x), A_\eps(y))\geq d/2$. 

 We choose  $\psi=\chi_{A_\eps(x)}$ and $\phi=\chi_{A_\eps(y)}$, and substituting   in \eqref{e5} we find, recalling that $K$ is decreasing, that 
\begin{align*} 
0&\leq \int_{\R^d}  \int_{\R^d }  (\psi(x)-\phi(x))(\psi(y)-\phi(y))K(|x-y|)dxdy \\ &=\int_{A_\eps(x)}\int_{A_\eps(x)}K(|t-z|)dtdz+\int_{A_\eps(y)}\int_{A_\eps(y)}K(|t-z|)dtdz-2 \int_{A_\eps(x)}\int_{A_\eps(y)}K(|t-z|)dtdz\\ & \geq K(2\eps)|A_\eps(x)|^2+K(2\eps)|A_\eps(y)|^2-2 K(d/2) |A_\eps(x)||A_\eps(y)|\\ & =2 |A_\eps(x)|^2(K(2\eps)- K(d/2))>0, \end{align*} 
which gives a contradiction.
\end{proof} 

We provide now the existence and characterization of minimizers to $\cW$.  
\begin{theorem}\label{thmminima} Under the   assumptions \eqref{assK},  \eqref{assw},  the problem
\begin{equation}\label{minsta} \min_{m\in \cPr} \mathcal{W}(m)
\end{equation} admits at least one solution. Each solution  is given by $\rho\chi_E$ for some measurable set $E$ such that $|E|=\rho^{-1}$.  Moreover if $\rho\chi_E$ is a minimizer then also
$\rho\chi_{\gamma E}$ is a minimizer. 

\end{theorem} 
\begin{proof} The result is an application of the direct method in calculus of variations. By Remark \ref{remb}, $\inf_{\cPr}\cW\geq -\mathcal{I}(\rho\chi_{B_{r_\rho}})$. 
Let  $m_n$ be a minimizing sequence. By \eqref{b} and   Lemma \ref{equiconv},, up to a subsequence, there exists $m$ such that $m\in \mathcal{P}_p(\R^d)$ for every $p<2$ such that  $m_n$ converges to $m$ narrowly and also weak-* in $L^\infty$. Again by the growth condition \eqref{b} and the lower semicontinuity property \eqref{lsc},  $m\in \mathcal{P}_2(\R^d)$.
Moreover  $\|m\|_\infty\leq \liminf_n \|m_n\|_\infty\leq \rho$, and so $m\in\cPr$ and, again by Proposition \ref{propcontK},
$\liminf_n \mathcal{W}(m_n)\geq \mathcal{W}(m)$, which implies that $m$ is a minimizer. 
Finally by Proposition \ref{propminima}, $m=\rho\chi_E$ for some bounded measurable set $E$. 
The fact that $\rho\chi_{\gamma E}$ is still a minimizers comes from \eqref{ref}. 
\end{proof}
\begin{proposition}\label{classiminimi}
Assume \eqref{assw}, \eqref{ref} and  \eqref{ass3}.  Then  $\min_{ \cPr } \cW=-\mathcal{I}(\rho\chi_{B_{r_\rho}})$ for $r_\rho=(\omega_d \rho)^{-1/d}$,  and   all the minimizers of \eqref{sta} are given by $\mathcal{M}^+\cup\mathcal{M}^-$, where $\mathcal{M}^-=\gamma\mathcal{M}^+$ and 
\[\mathcal{M}^+=\{\rho \chi_{E},  \text{ where } E= B(x', (\omega_d \rho)^{-1/d})\subseteq B(a^+,\tilde r) \text{ for some $x'\in\R^d$}\}. \]
If $\rho=\tilde \rho \,\,\big(= (\omega_d \tilde r^d)^{-1}\big)$, then $\mathcal{M}^+=\{\rho\chi_{B(a^+, \tilde r)}\}$. 

Moreover, $\mathcal{M}^+$ and $ \mathcal{M}^-$ are compact subsets of $\mathcal{P}_2(\R^d)$ and $d_2(\mathcal{M^+}, \mathcal{M^-}) > 0$. 
\end{proposition}
\begin{proof}
Let $\rho\chi_E$ be a minimizer. We get, recalling Remark \ref{remb}, 
\[ -\mathcal{I}(\rho\chi_{B(0,r_\rho)}) \leq \rho \int_{E} W(x)dx -\mathcal{I}(\rho\chi_E) \leq \rho \int_{B(a^+, r_\rho)} W(x)dx-\mathcal{I}(\rho\chi_{B(0, r_\rho)}). \] 
 Note that under assumption \eqref{ass3}, $\omega_d r_\rho^d\leq \omega_d\tilde r^d=|B(a^+,\tilde r)|$, then $\cW(\rho\chi_{B(a^\pm, r_\rho)})= -\mathcal{I}(\rho\chi_{B(0,r_\rho)})$,  and so $\rho \chi_{B(a^\pm, r_\rho)}$ are minimizers. Moreover, due to Lemma \ref{cKbounds},   $\int_E W(x)dx=0$  for every $E$ such that $\rho \chi_E$ is a minimizer.  If $r_\rho<\tilde r$, there are infinitely many minimizers, which are given by all possible balls $B(x', r_\rho)\subseteq B(a^\pm, \tilde r)$, whereas if $r_\rho=\tilde r$, the only minimizers are $B(a^\pm, \tilde r)$. 
 
 The compactness of $\mathcal{M}^\pm$ is straightforward. To evaluate $d_2(\mathcal{M^+}, \mathcal{M^-})$, we just notice that for any $\bar m^+\in \mathcal{M}^+$ and $ \bar m^-\in \mathcal{M}^-$, $\bar m^\pm = \rho \chi_{E^\pm}$, where $E^+= B(x',  r_\rho)\subseteq B(a^+,r)$ and $E^-= B(y',  r_\rho)\subseteq B(a^-,\tilde r)$. Therefore, for all $\pi \in \Pi(\bar m^+ dx, \bar m^- dx)$,
 \[
 \int_{\R^d} \int_{\R^d} |x-y|^2 d\pi(x,y) = \int_{B(a^+,\tilde r)} \int_{B(a^-,\tilde r)} |x-y|^2 d\pi(x,y) \ge [2(|a^+| - \tilde r)]^2 > 0,
 \]
 since $B(a^\pm, \tilde r)$ are symmetric and disjoint. Hence,
 \[
 d_2(\mathcal{M^+}, \mathcal{M^-}) \ge d_2(\bar m^+,\bar m^-) \ge [2(|a^+| - \tilde r)]^2 > 0.
 \]
% 
% 
% 
% we make use of the standard duality formula for $d_1$ (note that elements of $\mathcal{M}^\pm$ have bounded support), see e.g. \cite[equation (7.1.2)]{ags}. Namely, for any $\bar m^+\in \mathcal{M}^+$ and $ \bar m^-\in \mathcal{M}^-$, $\bar m^\pm = \rho \chi_{E^\pm}$, where $E^+= B(x',  r_\rho)\subseteq B(a^+,r)$ and $E^-= B(y',  r_\rho)\subseteq B(a^-,r)$, and
% \begin{equation}\label{d2d1}
% d_2(\bar m^+,\bar m^-)\geq d_1(\bar m^+,\bar m^-) = \sup \left\{\int_{\R^d} \varphi d(\bar m^+ - \bar m^- ) \ \Big| \ \varphi: \R^d \to \R \text{ \ 1-Lipschitz}\right\}.
% \end{equation}
%Assume that the reflection $\gamma$ is given by $\gamma(x_1, x_2, \cdots, x_d) = (-x_1, x_2, \cdots, x_d)$ (the general case can be treated analogously). Since $B(a^+, \tilde r) \cap B(a^-, \tilde r) = \emptyset$ and $(a^+)_1 = -(a^-)_1$, assuming without loss of generality that $(a^+)_1 > 0$,  we have $(a^+)_1 - \tilde r > 0$ (otherwise $B(a^+, \tilde r)$ and $B(a^-, \tilde r)$ would have non-empty intersection). Picking $\varphi(x) = x_1$ in \eqref{d2d1} gives
%\begin{multline*}
% d_2(\bar m^+,\bar m^-)\geq \int_{B(x',  r_\rho)}x_1 \, dx - \int_{B(y',  r_\rho)}x_1 \, dx = |B(x',  r_\rho)| x'_1 - |B(y',  r_\rho)| y'_1 = \rho^{-1}(x'_1 - y'_1) \\
% \ge 2\rho^{-1}((a^+)_1 - \tilde r + r_\rho),
%\end{multline*}
%since $B(x',  r_\rho)\subseteq B(a^+,r)$ and $ B(y',  r_\rho)\subseteq B(a^-,r)$, yielding the conclusion.
\end{proof} 

Therefore,
\[
\cW_0(m) = \cW(m) + \mathcal{I}(\rho\chi_{B_{r_\rho}}).
\]
It remains to prove the property \eqref{c} for $\cW_0$.% , up to a renormalization constant $\mathcal{I}(\rho\chi_{B_{r_\rho}})$ and to the  definition of $\mathcal{M}^\pm$ (see the next Proposition \ref{classiminimi}). 
\begin{proposition}\label{continu}  Assume that $m_k\in \cPr$ is such that $\lim_k \cW(m_k)=-\mathcal{I}(\rho\chi_{B_{r_\rho}})$, with the notation of Lemma \ref{cKbounds}. Then  up to passing to a subsequence there exists $m\in \cPr$ such that $\cW(m)= -\mathcal{I}(\rho\chi_{B_{r_\rho}})$ and $\lim_k d_2(m_k, m)=0$. 
\end{proposition} 
\begin{proof} By the growth condition \eqref{b}, we get that $m_k$ are uniformly bounded in $\cP_2(\R^d)$. So, by Lemma \ref{lemmaconv}, up to passing to a subsequence we get that there exists $m\in \cP_p(\R^d)$ such that 
$m_k\to m$ in $\cP_p(\R^d)$ for all $p<2$, and a posteriori by the growth condition \eqref{assw},  and the lower semicontinuity in Proposition \ref{propcontK},   $m\in \cP_2(\R^d)$.  Again by the lower semicontinuity in   Proposition \ref{propcontK} we have that  $\cW(m)\leq -\mathcal{I}(\rho\chi_{B_{r_\rho}})$, and so, recalling Remark \ref{remb}, $\cW(m)=-\mathcal{I}(\rho\chi_{B_{r_\rho}})$. Since, by Proposition \ref{propcontK},  $\lim_k\mathcal{I}(m_k)=\mathcal{I}(m)$ this implies in particular that $\lim_n \int_{\R^d} W(x)m_k(dx)=\int_{\R^d}W(x)m(dx)$, so by Lemma \ref{remw}, $\lim_k d_2(m_k, m)=0$. \end{proof} 

The proof of Theorem \ref{teointro1} then follows by Theorem \ref{thmminima} and Proposition \ref{classiminimi}. Moreover, we obtain the following conclusion:
\begin{corollary}
Under the assumptions \eqref{assw},  \eqref{ref}, \eqref{ass3}, \eqref{assK}, \eqref{pos}, $\cW_0$ satisfies \eqref{z}, \eqref{lsc},  \eqref{b}, \eqref{c} and \eqref{rif}.
\end{corollary}

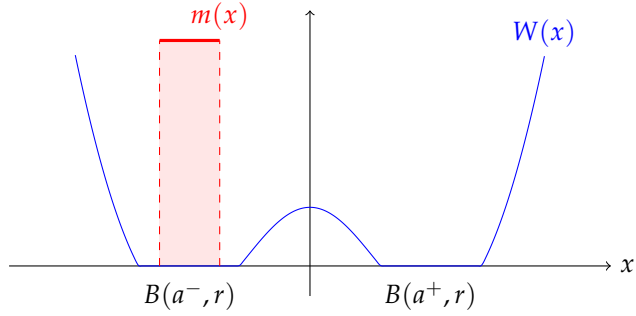
\begin{figure}

\centering
 \begin{tikzpicture}[scale=2]
      \fill [red, opacity=0.1] (-1,0) rectangle (-.6,1.5);
       \draw[->] (-2,0) -- (2,0) node[right] {$x$};
      \draw[->] (0,-.2) -- (0,1.7);
      \draw[scale=1.3,domain=-1.2:1.2,smooth,samples=200,variable=\x,blue] plot ({\x},{max(0, .3 - 3*\x*\x*(1-\x*\x)/(1+ \x*\x)) }) node[above] {$W(x)$};
      \draw[-,red, very thick] (-1,1.5) -- (-.6,1.5) node[above] {$m(x)$};
      \draw[-, red, dashed] (-.6,0) -- (-.6,1.5);
      \draw[-, red, dashed] (-1,0) -- (-1,1.5);
     \node at (.8,-.2) {$B(a^+, r)$};
       \node at (-.8,-.2) {$B(a^-, r)$};
 \end{tikzpicture}
 
 \caption{\footnotesize An example of symmetric potential $W$ and a (compactly supported) minimizer $m$ of $\cW$. }\label{fig2}
 
 \end{figure}
 
\small
\addcontentsline{toc}{section}{References}
\bibliography{etero}
\bibliographystyle{abbrv}

\medskip

\begin{flushright}
\noindent \verb"annalisa.cesaroni@unipd.it"\\
Dipartimento di Scienze Statistiche\\ Universit\`a di Padova\\
Via Battisti 241/243, 35121 Padova (Italy)

\smallskip

\noindent \verb"cirant@math.unipd.it"\\
Dipartimento di Matematica ``Tullio Levi-Civita" \\ Universit\`a di Padova\\
Via Trieste 63, 35121 Padova (Italy)
\end{flushright}
\end{document}